\documentclass{scrartcl}

\usepackage{microtype}

\addtokomafont{sectioning}{\normalfont\bfseries}
\setkomafont{title}{\normalfont\LARGE}
\setkomafont{subtitle}{\normalfont\Large}
\addtokomafont{pageheadfoot}{\scshape\small}

\usepackage[draft]{fixme}
\usepackage{pdfsync} 
\usepackage[english]{babel} 
\usepackage[utf8]{inputenc}
\usepackage[T1]{fontenc} 
\usepackage[usenames,dvipsnames]{xcolor} 
\usepackage{rotfloat} 
\usepackage{pdfpages} 
\usepackage{graphicx} 
\usepackage[usenames,dvipsnames]{xcolor}
\usepackage[format=plain, indention=0.2cm]{caption}
\usepackage{url} 
\usepackage{booktabs} 
\usepackage{textcomp} 
\usepackage{marvosym} 
\usepackage{multirow} 
\usepackage{makecell} 
\usepackage[]{algorithm2e}
\usepackage{amsfonts, amsmath, amsthm, amssymb, stmaryrd}
\usepackage{mathrsfs}
\usepackage{upgreek}
\usepackage{todonotes}
\usepackage{nicefrac}
\usepackage{txfonts}
\usepackage{csquotes}

\DeclareMathOperator{\supp}{supp}
\DeclareMathOperator{\Span}{span}

\DeclareMathOperator{\divergence}{div}

\newtheorem{definition}{Definition}
\theoremstyle{plain}
\newtheorem{theorem}[definition]{Theorem}
\theoremstyle{plain}
\newtheorem{remark}[definition]{Remark}
\theoremstyle{plain}
\newtheorem{lemma}[definition]{Lemma}
\theoremstyle{plain}

\usepackage{bbold}

\title{Variational discretization approach applied to an optimal control problem with bounded measure controls}
\author{Evelyn Herberg\footnote{Mathematisches Institut, Universität Koblenz-Landau, Campus Koblenz, Universitätsstraße 1, 56070 Koblenz, Germany.}$\,$ , Michael Hinze$^*$}
\date{August 17, 2020}

\begin{document}
	
	\maketitle
	
	\textbf{Abstract.} We consider a parabolic optimal control problem with an initial measure control. The cost functional consists of a tracking term corresponding to the observation of the state at final time. Instead of a regularization term in the cost functional, we follow \cite{CK19} and consider a bound on the measure norm of the initial control. The variational discretization of the problem together with the optimality conditions induce maximal discrete sparsity of the initial control, i.e. Dirac measures in space. We present numerical experiments to illustrate our approach.\\
	
	
\section{Introduction}
\label{sec:intro}
We consider the following optimal control problem which was analyzed in \cite{CK19}:

\begin{equation}
\min_{u \in U_{\alpha}} J(u) = \frac{1}{2} \| y_u(T) - y_d\|^2_{L^2(\Omega)}.
\tag{$P_{\alpha}$}
\label{eq:Palpha}
\end{equation}
Here let $y_d \in L^2(\Omega)$, and $U_{\alpha} := \{u \in \mathcal{M}(\bar{\Omega}) : \|u\|_{\mathcal{M}(\bar{\Omega})} \leq \alpha \}$, where $\mathcal{M}(\bar{\Omega})$ denotes the space of regular Borel measures on $\bar{\Omega}$ equipped with the norm
\begin{equation*}
\|u\|_{\mathcal{M}(\bar{\Omega})} := \sup_{\|\phi\|_{C(\bar{\Omega})} \leq 1 } \int_{\bar{\Omega}} \phi(x) \, d  u(x) = |u|(\bar{\Omega}).
\end{equation*}
The state $y_u$ solves the parabolic equation 
\begin{equation}
\begin{cases}
\partial_t y_u + A y_u &= f, \qquad \text{in} \; Q = \Omega \times (0,T),\\
y_u(x,0) &=u, \qquad \text{in} \; \bar{\Omega}, \\
\partial_n y_u(x,t) &= 0, \qquad \text{on} \; \Sigma = \Gamma \times (0,T),
\end{cases}
\label{eq:PDE}
\end{equation}
where $f \in L^1(0,T; L^2(\Omega))$ is given, $\Omega \subset \mathbb{R}^n (n =1,2,3)$ denotes an open, connected and bounded set with Lipschitz boundary $\Gamma$, and $A$ is the elliptic operator defined by 
\begin{equation}
A y_u := -a \Delta y_u + b(x,t) \cdot \nabla y_u + c(x,t) y_u,
\label{eq:operatorA}
\end{equation}
with a constant $a >0$ and functions $b \in L^{\infty}(Q)^n$ and $c\in L^{\infty}(Q)$. 

The state is supposed to solve \eqref{eq:PDE} in the following very weak sense, see e.g. \cite[Definition 2.1]{CK19}:
\begin{definition} 
	We say that a function $y \in L^1(Q)$ is a solution of \eqref{eq:PDE} if the following identity holds:
	\begin{equation} \label{eq:stateeq}
	\int_Q (-\partial_t \phi + A^* \phi) y \, dx  dt = \int_Q f \phi \, dx dt + \int_{\bar{\Omega}} \phi(0) \, d u  \quad \forall \, \phi \in \Phi,
	\end{equation}	
	where 
	\begin{equation*}
	\Phi := \{ \phi \in L^2(0,T;H^1(\Omega)) : -\partial_t \phi + A^* \phi \in L^{\infty}(Q), \partial_n \phi = 0 ~\textrm{on}~ \Sigma,\phi(T) =0 \in \Omega \}
	\end{equation*}
	and $A^* \bar{\varphi} := -a \Delta\bar{\varphi} - \divergence [b(x,t)\bar{\varphi} ] + c \bar{\varphi}  $ denotes the adjoint operator of $A$. 
\end{definition}

The existence and uniqueness of solutions to the state equation \eqref{eq:PDE} and problem \eqref{eq:Palpha} have been established in \cite[Theorem 2.2 and Theorem 2.4]{CK19}.

Optimal control with a bound on the total variation norm of the measure-control is inspired by applications, which aim at identifying pollution sources, see, e.g. \cite{EHH2005,LOT2014}. These problems inherit a sparsity structure (see, e.g., \cite{Gong,GongH,Stadler}), which we can retain in practical implementation by applying variational discretization, from \cite{VD} with a suitable Petrov-Galerkin approximation of the state equation \eqref{eq:stateeq}, compare \cite{HerbergHS}. 

Let us briefly comment on related contributions in the literature. In \cite{HerbergHS} the variational discrete approach is applied to an optimal control problem with parabolic partial differential equation and space-time measure control from \cite{CasKun}. Control of elliptic partial differential equations with measure controls is considered in \cite{sparseFEM,duality,CS2017,PV} and control of parabolic partial differential equations with measure controls can be found in \cite{CasasClasonKunisch,CasKun,CVZ,KPV,LVW2019}. The novelty of the problem discussed in this work, lies in constraining the control set, instead of incorporating a penalty term for the control in the target functional.

The plan of the paper is as follows: We analyze the continuous problem, its sparsity structure and the special case of positive controls in Section \ref{sec:cont}. Thereafter we apply variational discretization to the optimal control problem in Section \ref{sec:VD}. Finally in Section \ref{sec:Num} we apply the semismooth Newton method to the optimal control problem with positive controls (Subsection \ref{subsec:pos}) and to the original optimal control problem (Subsection \ref{subsec:gen}). For the latter we add a penalty term before applying the semismooth Newton method. For both cases we provide numerical examples.

\section{Continuous optimality system}
\label{sec:cont}

In this section we summarize properties of \eqref{eq:Palpha}, which have been established in \cite{CK19}.

Let $\bar{u}$ be the unique solution of \eqref{eq:Palpha} with associated state $\bar{y}$. We then say that \newline$\bar{\varphi}\in L^2(0,T;H^1(\Omega)) \cap C\left(\bar{\Omega} \times [0,T]\right)$ is the associated adjoint state of $\bar{u}$, if it solves 
\begin{equation}
\begin{cases}
-\partial_t \bar{\varphi} + A^* \bar{\varphi} &= 0, \qquad\qquad\qquad\, \text{in} \; Q,\\
\bar{\varphi}(x,T) &=\bar{y}(x,T) - y_d , \quad\;\, \text{in} \; \Omega, \\
\partial_n \bar{\varphi}(x,t) &= 0, \qquad\qquad\qquad\, \text{on} \; \Sigma .
\end{cases}
\label{eq:adjointPDE}
\end{equation}
We recall the optimality conditions for \eqref{eq:Palpha} from \cite[Theorem 2.5]{CK19}:
\begin{theorem} \label{thm:optcond}
	Let $\bar{u}$ be the solution of \eqref{eq:Palpha} with $\bar{y}$ and $\bar{\varphi}$ the associated state and adjoint state, respectively. Then, the following properties hold
	\begin{enumerate}
		\item If $\|\bar{u}\|_{\mathcal{M}(\bar{\Omega})} < \alpha$, then $\bar{y}(T) = y_d$ and $\bar{\varphi} = 0 \in Q$.
		\item If $\|\bar{u}\|_{\mathcal{M}(\bar{\Omega})} = \alpha$, then
		\begin{align*}
		\supp(\bar{u}^+) &\subset \{ x \in \bar{\Omega} : \bar{\varphi}(x,0) = - \|\bar{\varphi}(0)\|_{C(\bar{\Omega})} \}, \\
		\supp(\bar{u}^-) &\subset \{ x \in \bar{\Omega} : \bar{\varphi}(x,0) = + \|\bar{\varphi}(0)\|_{C(\bar{\Omega})}\},
		\end{align*}
		where $\bar{u} = \bar{u}^+ - \bar{u}^- $ is the Jordan decomposition of $\bar{u}$. 
	\end{enumerate} 
	Conversely, if $\bar{u}$ is an element of $U_{\alpha}$ satisfying 1. or 2., then $\bar{u}$ is the solution to \eqref{eq:Palpha}.\\
\end{theorem}
%
%
In some applications we may have a priori knowledge about the measure controls. This motivates the restriction of the admissible control set to positive controls $U^+_{\alpha} := \{u \in \mathcal{M}^+(\bar{\Omega}) : \|u\|_{\mathcal{M}(\bar{\Omega})} \leq \alpha \}$, with $\|u\|_{\mathcal{M}(\bar{\Omega})}  = u (\bar{\Omega})$. We then consider the problem
\begin{equation}
\min_{u \in U^+_{\alpha}} J(u) = \frac{1}{2} \| y_u(T) - y_d\|^2_{L^2(\Omega)},
\tag{$P^+_{\alpha}$}
\label{eq:Palphaplus}
\end{equation}
where $y_u$ solves \eqref{eq:PDE}.
The properties of \eqref{eq:Palphaplus} have been derived in \cite[Theorem 3.1]{CK19}:

\begin{theorem} \label{thm:optcondplus}
	\eqref{eq:Palphaplus} has a unique solution. Let $\bar{u}$ be the unique solution of \eqref{eq:Palphaplus} with associated adjoint state $\bar{\varphi}$. Then, $\bar{u}$ is a solution of \eqref{eq:Palphaplus} if and only if 
	\begin{equation}
	\int_{\bar{\Omega}} \bar{\varphi}(x,0) \, d \bar{u} \leq \int_{\bar{\Omega}} \bar{\varphi}(x,0) \, d u \qquad \forall \, u \in U^+_{\alpha}.
	\label{eq:intineq} 
	\end{equation}
	If $u(\bar{\Omega}) = \alpha$ the following properties are fulfilled:
	\begin{enumerate}
		\item Inequality \eqref{eq:intineq} is equivalent to the identity  
		\begin{equation}
		\int_{\bar{\Omega}} \bar{\varphi}(x,0) \, d \bar{u} = \alpha \bar{\lambda} := \alpha \min_{x \in \bar{\Omega}} \bar{\varphi}(x,0) ,
		\label{eq:lambdabar}
		\end{equation}
		where $\bar{\lambda} \leq 0$.
		\item $\bar{u}$ is the solution of \eqref{eq:Palphaplus} if and only if
		\begin{equation}
		\supp(\bar u) \subset \{ x \in \bar{\Omega} : \bar{\varphi}(x,0) = \bar{\lambda} \}.
		\label{eq:supppos}
		\end{equation} 
	\end{enumerate}
\end{theorem}
We also repeat the following remark from \cite[Remark 3.3]{CK19}:
\begin{remark}
	\label{remark}
	While in Theorem \ref{thm:optcond}, we have $\bar{y}(T) = y_d$ and $\bar \varphi = 0 \in Q$ for an optimal control $\bar u$ with $\bar u(\bar{\Omega}) < \alpha$, this case is not a part of Theorem \ref{thm:optcondplus}. For non-negative controls we can show that if $y_d \leq y_0(T) $, where by $y_0$ we denote the solution of \eqref{eq:PDE} corresponding to the control $u = 0$, then the unique solution to \eqref{eq:Palphaplus} is given by $\bar{u} = 0$. So even though $\bar u(\bar{\Omega}) = 0 < \alpha$, we have $\bar{y}(T) \neq y_d$ and consequently $\bar \varphi \neq 0 \in Q$. 
\end{remark}

\section{Variational discretization}
\label{sec:VD}

To discretize problems \eqref{eq:Palpha}, \eqref{eq:Palphaplus} we define the space-time grid as follows: Define the partition $0 =t_0 < t_1 <\ldots <t_{N_{\tau}} =T$. For the temporal grid the interval $I$ is split into subintervals $I_k = \left( t_{k-1},t_k \right]$ for $k=1,\ldots,N_{\tau}$. The temporal gridsize is denoted by $\tau = \max_{0\leq k \leq N_{\tau}} {\tau_k} $, where $\tau_k := t_{k}-t_{k-1}$. We assume that $\{I_k\}_k$ and $\{\mathcal K_h\}_h$ are quasi-uniform sequences of time grids and triangulations, respectively. For $K \in \mathcal K_h$ we denote by $\rho(K)$ the diameter of $K$, and $h:= \max_{K\in \mathcal K_h}\rho(K)$.
We set $\bar{\Omega}_h= \bigcup_{K \in \mathcal{K}_h}K$ and denote by $\Omega_h$ the interior and by $\varGamma_h$ the boundary of $\bar{\Omega}_h$. We assume that vertices on $\varGamma_h$ are points on $\varGamma$. We then set up the space-time grid as $Q_h := \Omega_h \times (0,T)$.

We define the discrete spaces: 
\begin{align}
Y_h &:= \Span \{ \phi_j :   1 \leq j \leq N_h \}, \\
Y_{\sigma} &:= \Span \{ \phi_j \otimes \chi_k : 1 \leq j \leq N_h, 1 \leq k \leq N_{\tau} \},
\end{align}
where $\chi_k$ is the indicator function of $I_k$ and $\left(\phi_j\right)_{j=1}^{N_h}$ is the nodal basis formed by continuous piecewise linear functions satisfying $\phi_j(x_i)=\updelta_{ij}$. 

We choose the space $Y_{\sigma}$ as our discrete state and test space in a dG(0) approximation of \eqref{eq:PDE}. The control space remains either $U_{\alpha}$ or $U_{\alpha}^+$. 


This approximation scheme is equivalent to an implicit Euler time stepping scheme. To see this we recall that the elements $y_{\sigma} \in Y_{\sigma}$ can be represented as
\begin{equation*}
y_{\sigma} = \textstyle \sum_{k=1}^{N_{\tau}} y_{k,h} \otimes \chi_k,
\end{equation*}
with $y_{k,h} := y_{\sigma}|_{I_k} \in Y_h$. 

Given a control $u \in U_{\alpha}$ for $k = 1, \ldots ,N_{\tau}$ and $z_h \in Y_h$ we thus end up with the variational discrete scheme
\begin{equation}
\begin{cases}
\left( y_{k,h}-y_{k-1,h} , z_h \right)_{L^2} + a \, \tau_k \int_{\Omega}{\nabla y_{k,h} \nabla z_h \, dx}   \\
\quad +\int_{I_k}\int_{\Omega} b(x,t) \nabla y_{k,h} \, z_h \, + c(x,t) y_{k,h} \, z_h \, dx\, dt  =  \int_{I_k}\int_{\Omega} f\, z_h \, dx\, dt, \\
y_{0,h}=y_{0h},
\end{cases}
\label{eq:dse}
\end{equation}
where $y_{0h} \in Y_h$ is the unique element satisfying:
\begin{equation}
(y_{0h},z_h) = \int_{\Omega}{z_h \, d u} \qquad \forall \, z_h \in Y_h.\label{eq:y0h}\\
\end{equation}
Here $(\cdot,\cdot)_{L^2}$ denotes the $L^2(\Omega)$ inner product. We assume that the discretization parameters $h$ and $\tau$ are sufficiently small, such that there exists a unique solution to \eqref{eq:dse} for general functions $b$ and $c$.\\

The variational discrete counterparts to \eqref{eq:Palpha} and \eqref{eq:Palphaplus} now read
\begin{equation}
\min_{u \in U_{\alpha}} J_{\sigma}(u) = \frac{1}{2} \| y_{u,\sigma}(T) - y_d\|^2_{L^2(\Omega_h)},
\tag{$P_{\alpha,\sigma}$}
\label{eq:Psigma}
\end{equation}
and 
\begin{equation}
\min_{u \in U_{\alpha}^+} J_{\sigma}(u) = \frac{1}{2} \| y_{u,\sigma}(T) - y_d\|^2_{L^2(\Omega_h)},
\tag{$P_{\alpha,\sigma}^+$}
\label{eq:Psigmaplus}
\end{equation}
respectively, where in both cases $y_{u,\sigma}$ for given $u$ denotes the unique solution of \eqref{eq:dse}.  It is now straightforward to show that the optimality conditions for the problems \eqref{eq:Psigma} and \eqref{eq:Psigmaplus} read like those for \eqref{eq:Palpha} and \eqref{eq:Palphaplus} with the adjoint $\varphi$ replaced by $\varphi_{\bar u,\sigma} \in Y_h$ for given solution $\bar u$, the solution to the following system for $k = 1, \ldots, N_{\tau}$ and $z_h \in Y_h$ :
\begin{equation}
\begin{cases}
- \left( \varphi_{k,h}-\varphi_{k-1,h} , z_h \right)_{L^2} + a \, \tau_k \int_{\Omega}{\nabla \varphi_{k-1,h} \nabla z_h \, dx}   \\
\quad +\int_{I_k}\int_{\Omega} - \divergence( b(x,t) \varphi_{k-1,h}) \, z_h \, + c(x,t) \varphi_{k-1,h} \, z_h \, dx\, dt  =  0,\\
\varphi_{N_{\tau},h}=\varphi_{N_{\tau} h},
\end{cases}
\label{eq:adjdse}
\end{equation}
where $z_h \in Y_h$ and $\varphi_{N_{\tau} h} \in Y_h$ is the unique element satisfying:
\begin{equation}
(\varphi_{N_{\tau} h},z_h) = \int_{\Omega}{ (y_{\bar u, \sigma}(T)-y_d) z_h \, d x} \qquad \forall \, z_h \in Y_h.\label{eq:varphiNth}\\
\end{equation}
For details on the derivation of the optimality conditions we refer to Theorem \ref{thm:optconddisc} and Theorem \ref{thm:optconddiscplus}, which will be proven after introducing a few helpful results.

This in particular implies that
\begin{align}\label{eq:suppdiskret}
\supp(\bar{u}^+) &\subset \{ x \in \bar{\Omega} : \varphi_{\bar u,\sigma}(x,0) = - \|\varphi_{\bar u,\sigma}(0)\|_{\infty} \},  \\
\supp(\bar{u}^-) &\subset \{ x \in \bar{\Omega} : \varphi_{\bar u,\sigma}(x,0) = + \|\varphi_{\bar u,\sigma}(0)\|_{\infty} \}. \notag
\end{align}
Analogously for \eqref{eq:Psigmaplus}, in the case $u(\bar \Omega) = \alpha$ we have the optimality condition
\begin{equation*}
\supp(\bar u) \subset \{ x \in \bar{\Omega} : \varphi_{\bar u,\sigma}(x,0) = \min_{x \in \bar{\Omega}} \varphi_{\bar u,\sigma}(x,0)  \}.
\end{equation*}
Since, in both cases, $\varphi_{\bar u,\sigma}$ is a piecewise linear and continuous function, the extremal value in the generic case can only be attained at grid points, which leads to 
\begin{equation*}
\supp(\bar{u}) \subset \{x_j\}_{j=1}^{N_h}.
\end{equation*}
So, we derive the implicit discrete structure:
\begin{equation*}
\bar{u} \in U_h := \Span \{ \delta_{x_j} : 1 \leq j \leq N_h \},
\end{equation*}
where $\delta_{x_j}$ denotes a Dirac measure at gridpoint $x_j$. In the case of \eqref{eq:Psigmaplus} we even know that all coefficients will be positive and hence we get
\begin{equation*}
\bar u \in U_h^+ := \left\{ \textstyle\sum_{j=1}^{N_h} u_j \delta_{x_j} : u_j \geq 0  \right\}.
\end{equation*}
Notice also that the natural pairing $\mathcal{M}(\bar \Omega) \times \mathcal{C}(\bar \Omega) \rightarrow \mathbb{R}$ induces the duality $Y_h^* \cong U_h$ in the discrete setting. Here we see the effect of the variational discretization concept: The choice for the discretization of the test space induces a natural discretization for the controls.

We note that the use of piecewise linear and continuous Ansatz- and test-functions in the variational discretization creates a setting, where the optimal control is supported on space grid points. However, it is possible to use piecewise quadratic and continuous Ansatz- and test-functions, so that the discrete adjoint variable can attain its extremal values not only on grid points, but anywhere. Calculating the location of these extremal values, then, would mean to determine the potential support of the optimal control - not limited to grid points anymore. 

The following operator will be useful for the discussion of solutions to \eqref{eq:Psigma}.
\begin{lemma}
	Let the linear operator $\Upsilon_h$ be defined as below:
	\begin{equation*}
	\Upsilon_h: \mathcal{M}(\bar \Omega) \rightarrow U_h \subset \mathcal{M}(\bar \Omega), \qquad \Upsilon_h u := \sum_{j =1}^{N_h} \delta_{x_j} \int_{{\Omega}} \phi_j \, d u.
	\end{equation*}
	Then for every $u \in \mathcal{M}(\bar \Omega) $ and $\varphi_h \in Y_h$ the following properties hold.
	\begin{align}
	\left\langle u, \varphi_h\right\rangle &= \left\langle \Upsilon_h u, \varphi_h\right\rangle, \label{eq:proj1}\\
	\| \Upsilon_h u \|_{\mathcal{M}(\bar \Omega)} &\leq \|  u \|_{\mathcal{M}(\bar \Omega)}. \label{eq:proj2}
	\end{align}
\end{lemma}
These results have been proven in \cite[Proposition 4.1.]{CasKun}. Furthermore, it is obvious, for piecewise linear and continuous finite elements, that $\Upsilon_h(\mathcal{M}^+(\bar \Omega)) \subset U_h^+$.


The mapping $u \mapsto y_{u,\sigma}(T)$ is in general not injective, hence the uniqueness of the solution cannot be concluded. In the implicitly discrete setting however, we can prove uniqueness similarly as done in \cite[Section 4.3.]{CasasClasonKunisch} and \cite[Theorem 11]{HerbergHS}.
\begin{theorem}
	\label{thm:solPsigma}
	The problem \eqref{eq:Psigma} has at least one solution in $\mathcal{M}(\bar \Omega)$ and there exists a unique solution $\bar u \in U_h$. Furthermore, for every solution $\hat u \in \mathcal{M}(\bar \Omega)$ of \eqref{eq:Psigma} it holds $\Upsilon_h \hat u = \bar u $. Moreover, if $\bar \varphi_h(x_j) \neq \bar \varphi_h(x_k)$ for all neighboring finite element nodes $x_j \neq x_k$ of the finite element nodes $x_j (j = 1,\ldots,N_h)$, problem \eqref{eq:Psigma} admits a unique solution, which is an element of $U_h$. 
\end{theorem}
\begin{proof}
	The existence of solutions can be derived as for the continuous problem, see \cite[Theorem 2.4.]{CK19}, since the control domain remains continuous. We include the details for the convenience of the reader.
	
	The control domain $U_{\alpha}$ is bounded and weakly-* closed in $\mathcal{M}(\bar \Omega)$. From Banach-Alaoglu-Bourbaki theorem we even know that it is weakly-* compact, see e.g. \cite[Theorem 3.16.]{MR2759829}. Hence, any minimizing sequence is bounded in $\mathcal{M}(\bar \Omega)$ and any weak-* limit belongs to the control domain $U_{\alpha}$. Using convergence properties from \cite[Theorem 2.3.]{CK19} we can conclude that any of these limits is a solution to \eqref{eq:Psigma}.
	
	Let $\hat u \in \mathcal{M}(\bar \Omega)$ be a solution of \eqref{eq:Psigma} and $\bar u := \Upsilon_h \hat u \in U_h$. From \eqref{eq:proj1} we have 
	\begin{equation*}
	y_{u,\sigma} = y_{\Upsilon_h u, \sigma} \quad \text{ for all } u \in \mathcal{M}(\bar \Omega).
	\end{equation*}
	From this we deduce $J_{\sigma}(\bar u) = J_{\sigma}(\hat u)$. Moreover \eqref{eq:proj2} delivers 
	\begin{equation*}
	\| \bar u \|_{\mathcal{M}(\bar \Omega)} \leq \| \hat u \|_{\mathcal{M}(\bar \Omega)},
	\end{equation*}
	so $\bar u$ is admissible, since $\hat u \in U_{\alpha}$. Altogether, this shows the existence of solutions in the discrete space $U_h$.

	Since the mapping $u \mapsto y_{u,\sigma}(T)$ is injective for $u \in U_h$ - since we have $\dim(U_h) = \dim(Y_h)$, and $J_{\sigma}(u)$ is a quadratic function, we deduce strict convexity of $J_{\sigma}(u)$ on $U_h$. Furthermore $\left\{ u \in U_h : \|u\|_{\mathcal{M}(\bar \Omega)} = \sum_{j=1}^{N_h}|u_j| \leq \alpha \right\}$ is a closed and convex set, so we can conclude the uniqueness of the solution in the discrete space.
	
	For every solution $\hat u \in \mathcal{M}(\bar \Omega)$ of \eqref{eq:Psigma}, the projection $\Upsilon_h \hat u$ is a discrete solution. Moreover, there exists only one discrete solution. So we deduce that all projections must coincide.
	
	If now $\bar\varphi_h(x_j)\neq \bar\varphi_h(x_k)$ for all neighbors $k\neq j$, every solution $u$ of \eqref{eq:Psigma} has its support in some of the finite element nodes of the triangulation, or vanish identically, and thus is an element of $U_h$. This shows the unique solvability of \eqref{eq:Psigma} in this case.

\end{proof}

\begin{remark}
	We note that the condition on the values of $\bar \varphi_h$ in the finite element nodes for guaranteeing uniqueness can be checked once the discrete adjoint solution is known. This condition is thus fully practical.
\end{remark}

For \eqref{eq:Psigmaplus} we have a similar result like Theorem \ref{thm:solPsigma}, which we state without proof, since it can be interpreted as a special case of Theorem \ref{thm:solPsigma} and can be proven analogously.
\begin{theorem}
	The problem \eqref{eq:Psigmaplus} has at least one solution in $\mathcal{M}^+(\bar \Omega)$ and there exists a unique solution $\bar u \in U_h^+$. Furthermore, for every solution $\hat u \in \mathcal{M}^+(\bar \Omega)$ of \eqref{eq:Psigmaplus} it holds $\Upsilon_h \hat u = \bar u $. Moreover, if $\bar \varphi_h(x_j) \neq \bar \varphi_h(x_k)$ for all neighboring finite element nodes $x_j \neq x_k$ of the finite element nodes $x_j (j = 1,\ldots,N_h)$, problem \eqref{eq:Psigmaplus} admits a unique solution, which is an element of $U_h^+$. 
\end{theorem}

Now, we introduce two useful lemmas. 

\begin{lemma}
	\label{lem:2.6}
	Given $u \in \mathcal{M}(\bar \Omega)$, the solution $z_{u,\sigma} \in Y_h$ to \eqref{eq:dse} with $f \equiv 0$ satisfies
	\begin{equation}
	\label{eq:integraleq}
	\int_{\Omega} (y_{u,\sigma}(T)-y_d)z_{u,\sigma}(T) \, dx = \int_{ \Omega} \varphi_{u, \sigma}(0) \, du.
	\end{equation}
\end{lemma} 

\begin{proof}
	We take \eqref{eq:dse} with $f \equiv 0$ and test with $\varphi_{k,h}$, the components of $\varphi_{u,\sigma}$, for all $k = 1,\ldots,N_{\tau}$. Similarly we take \eqref{eq:adjdse} and test this with $z_{k-1,h}$, the components of $z_{u,\sigma}$, for all $k = 1,\ldots,N_{\tau} $. Now we can sum up the equations, and since in both cases the right hand side is zero, we can equalize those sums. Furthermore, we can apply Gauß' theorem and drop all terms that appear on both sides. This leads to
	\begin{alignat}{2}
	& \qquad\qquad\;\;\; \sum_{k=1}^{N_{\tau}} (z_{k,h}-z_{k-1,h}, \varphi_{k,h}) &&= \sum_{k=1}^{N_{\tau}} (-\varphi_{k,h}+\varphi_{k-1,h}, z_{k-1,h}) , \notag\\
	&\Rightarrow\; \sum_{k=1}^{N_{\tau}} (z_{k,h}, \varphi_{k,h}) - (z_{k-1,h}, \varphi_{k,h}) &&= \sum_{k=1}^{N_{\tau}} -(\varphi_{k,h}, z_{k-1,h}) + (\varphi_{k-1,h}, z_{k-1,h}) , \notag\\
	&\Rightarrow \qquad\qquad\qquad\quad\;\; \sum_{k=1}^{N_{\tau}} (z_{k,h}, \varphi_{k,h}) &&= \sum_{k=0}^{N_{\tau}-1} (\varphi_{k,h}, z_{k,h}), \notag \\
	&\Rightarrow \qquad\qquad\qquad\quad\;\;\, (z_{N_{\tau},h}, \varphi_{N_{\tau},h}) &&= (\varphi_{0,h},z_{0,h}). \notag
	\end{alignat}
	We have $z_{N_{\tau},h} = z_{u,\sigma}(T) \in Y_h$ and $\varphi_{0,h} = \varphi_{u,\sigma}(0) \in Y_h$, so together with \eqref{eq:y0h} and \eqref{eq:varphiNth} we can deduce \eqref{eq:integraleq}.
\end{proof}

\begin{lemma} \label{lem:2.7}
	For every $\epsilon >0$ and $h$ small enough, there exists a control $u \in L^2(\Omega)$, such that the solution $y_{u,\sigma}$ of \eqref{eq:dse} fulfills
	\begin{equation}
	\| y_{u,\sigma}(T) - y_d \|_{L^2(\Omega_h)} < \epsilon .
	\end{equation}
\end{lemma}

\begin{proof}
	Let $y_{d,\sigma}$ be the $L^2$-projection of $y_d$ onto $Y_h$, then for $h$ small enough 
	\begin{equation*}
	\| y_{u,\sigma}(T) - y_d \|_{L^2(\Omega_h)} \leq \| y_{u,\sigma}(T) - y_{d,\sigma} \|_{L^2(\Omega_h)} + \underbrace{\| y_{d,\sigma} -y_d \|_{L^2(\Omega_h)}}_{< \epsilon}.
	\end{equation*}
	
	Let additionally $\tau$ be small enough, such that the scheme \eqref{eq:dse} has a unique solution. Then in every time-step we obtain a system of equations, where the matrix is an isomorphism on $Y_h$. Consequently the initial to final value map $y_{0h} \mapsto y_{u,\sigma}(T)$ is an isomorphism. Since $Y_h \subset L^2(\Omega_h)$ we can find $u \in L^2(\Omega)$, such that
	\begin{equation*}
	\| y_{u,\sigma}(T) - y_{d,\sigma} \|_{L^2(\Omega_h)} = 0,
	\end{equation*}
	which completes the proof.

\end{proof}

Finally, we give the discrete version of Theorem \ref{thm:optcond} and Theorem \ref{thm:optcondplus}. Both are proven very similarly to the continuous case, see \cite[Theorem 2.5 and Theorem 3.1]{CK19}.

\begin{theorem}
	\label{thm:optconddisc}
	Let $\bar u$ solve \eqref{eq:Psigma} with $y_{\bar u, \sigma}$ and $\varphi_{\bar u, \sigma}$ the associated discrete state and discrete adjoint state, respectively. Then for $\sigma$ small enough, 
	\begin{enumerate}
		\item if $\|\bar{u}\|_{\mathcal{M}(\bar{\Omega})} < \alpha$, then $y_{\bar u, \sigma}(T) = y_d$ and $\varphi_{\bar u, \sigma} = 0 \in Q$.
		\item if $\|\bar{u}\|_{\mathcal{M}(\bar{\Omega})} = \alpha$, then
		\begin{align} \label{eq:suppconddisc1}
		\supp(\bar{u}^+) &\subset \{ x \in \bar{\Omega} : \varphi_{\bar u, \sigma}(x,0) = - \|\varphi_{\bar u, \sigma}(0)\|_{C(\bar{\Omega})} \}, \\
		\supp(\bar{u}^-) &\subset \{ x \in \bar{\Omega} : \varphi_{\bar u, \sigma}(x,0) = + \|\varphi_{\bar u, \sigma}(0)\|_{C(\bar{\Omega})}\}, \label{eq:suppconddisc2}
		\end{align}
		where $\bar{u} = \bar{u}^+ - \bar{u}^- $ is the Jordan decomposition of $\bar{u}$. 
	\end{enumerate} 
	Conversely, if $\bar{u}$ is an element of $U_{\alpha}$ satisfying 1. or 2., then $\bar{u}$ is the solution to \eqref{eq:Psigma}.\\
\end{theorem}

\begin{proof}
	Let $u \in U_{\alpha}$ arbitrary and denote by $z_{(u-\bar u),\sigma}$ the solution to \eqref{eq:dse} with $f \equiv 0$ and $u$ replaced by $u - \bar u$. From Lemma \ref{lem:2.6} we get
	\begin{align*}
	\lim_{\rho \searrow 0} \frac{J_{\sigma}(\bar u + \rho (u-\bar u)) - J_{\sigma}(\bar u)}{\rho} &= \int_{\Omega} (y_{\bar u, \sigma}(T)-y_d) z_{(u-\bar u),\sigma}(T) \, d x \\
	&= \int_{\Omega} \varphi_{\bar u, \sigma}(0) \, d (u-\bar u).
	\end{align*}
	Since \eqref{eq:Psigma} is a convex problem, the following variational inequality is a necessary and sufficient condition for optimality of a control $\bar u \in U_{\alpha}$:
	\begin{alignat}{3}
	& \qquad\qquad\quad\, J_{\sigma}'(\bar u)(u - \bar u) &&= \int_{\Omega} \varphi_{\bar u, \sigma}(0) \, d (u-\bar u) \geq 0 \qquad &&\forall u \in U_{\alpha} ,\notag \\
	&\Rightarrow \qquad\;\, - \int_{\Omega} \varphi_{\bar u, \sigma}(0) \, d u &&\leq - \int_{\Omega} \varphi_{\bar u, \sigma}(0) \, d \bar u \qquad &&\forall u \in U_{\alpha}, \notag\\
	&\Rightarrow \quad \sup_{u \in U_{\alpha} } \int_{\Omega} \varphi_{\bar u, \sigma}(0) \, d u &&= - \int_{\Omega} \varphi_{\bar u, \sigma}(0) \, d \bar u, && \notag \\
	&\Rightarrow \qquad \alpha \| \varphi_{\bar u, \sigma}(0) \|_{\mathcal{C}(\bar \Omega)} &&= - \int_{\Omega} \varphi_{\bar u, \sigma}(0) \, d \bar u .&& \notag
	\end{alignat}
	Let $\| \bar u \|_{\mathcal{M}(\bar \Omega)} = \alpha$, then this is 
	\begin{equation}
	\label{eq:2.18}
	\| \bar u \|_{\mathcal{M}(\bar \Omega)} \| \varphi_{\bar u, \sigma}(0) \|_{\mathcal{C}(\bar \Omega)} = - \int_{\Omega} \varphi_{\bar u, \sigma}(0) \, d \bar u.
	\end{equation}
	We now may conclude as in \cite[Lemma 3.4]{CasasClasonKunisch} to obtain \eqref{eq:suppconddisc1} and \eqref{eq:suppconddisc2}. Also, if these conditions hold we get the equality \eqref{eq:2.18}, which is a necessary and sufficient condition for optimality of $\bar u$, so $\bar u$ solves \eqref{eq:Psigma}.
	
	Let us now study the case $\| \bar u \|_{\mathcal{M}(\bar \Omega )} < \alpha$. If $y_{\bar u, \sigma} = y_d$, then $J_{\sigma}(\bar u) = 0$ and since $J_{\sigma}(u) \geq 0$ for all $u \in U_{\alpha}$, we deduce that $\bar u$ solves \eqref{eq:Psigmaplus}. Now assume that $\bar u$ solves \eqref{eq:Psigmaplus} and $y_{\bar u, \sigma} \neq y_d$ holds. Then we have $J_{\sigma}(\bar u) >0$. From Lemma \ref{lem:2.7} we know that for $h$ small enough there exists an element $u \in \mathcal{M}(\bar \Omega)$, such that $J_{\sigma}(u) < J_{\sigma} (\bar u)$. Since $\bar u$ is a solution to \eqref{eq:Psigmaplus}, it must hold $u \notin U_{\alpha}$. Now take $\lambda \in \mathbb{R}$, such that
	\begin{equation}
	0 < \lambda < \min \left\{ \frac{\alpha - \|\bar u \|_{\mathcal{M}(\bar \Omega)}}{ \|u - \bar u \|_{\mathcal{M}(\bar \Omega)} } , 1 \right\}.
	\end{equation}
	Then $v := \bar u + \lambda (u - \bar u) \in U_{\alpha}$ and by convexity of $J_{\sigma}$ we get
	\begin{equation*}
	J_{\sigma}(v) = J_{\sigma}(\lambda u + (1-\lambda) \bar u)	\leq \lambda \underbrace{J_{\sigma}(u)}_{< J_{\sigma}(\bar u)} + (1-\lambda) J_{\sigma}(\bar u)  < J_{\sigma} (\bar u),
	\end{equation*}
	so that $\bar u \in U_{\alpha}$ can not be the solution of \eqref{eq:Psigmaplus}. Hence $y_{\bar u, \sigma} = y_d$ must hold and from \eqref{eq:varphiNth} we deduce $\varphi_{\bar u, \sigma} = 0$.
	
\end{proof}

\begin{theorem}
	\label{thm:optconddiscplus}
	Let $\bar{u}$ solve \eqref{eq:Psigmaplus} with associated discrete adjoint state $\varphi_{\bar u, \sigma}$. Then, $\bar{u}$ is a solution of \eqref{eq:Psigmaplus} if and only if 
	\begin{equation}
	\int_{{\Omega}} \varphi_{\bar u, \sigma}(x,0) \, d \bar{u} \leq \int_{{\Omega}} \varphi_{\bar u, \sigma}(x,0) \, d u \qquad \forall \, u \in U^+_{\alpha}.
	\label{eq:intineqdisc} 
	\end{equation}
	If $\bar u(\bar{\Omega}) = \alpha$ the following properties are fulfilled:
	\begin{enumerate}
		\item Inequality \eqref{eq:intineqdisc} is equivalent to the identity  
		\begin{equation}
		\int_{{\Omega}} \varphi_{\bar u, \sigma}(x,0) \, d \bar{u} = \alpha \bar{\lambda} := \alpha \min_{x \in \bar{\Omega}} \varphi_{\bar u, \sigma}(x,0) ,
		\label{eq:lambdabardisc}
		\end{equation}
		where $\bar{\lambda} \leq 0$.
		\item $\bar{u}$ is the solution of \eqref{eq:Palphaplus} if and only if
		\begin{equation}
		\supp(\bar u) \subset \{ x \in \bar{\Omega} : \varphi_{\bar u, \sigma}(x,0) = \bar{\lambda} \}.
		\label{eq:suppposdisc}
		\end{equation} 
	\end{enumerate}
\end{theorem}

\begin{proof}
	As in the proof of Theorem \ref{thm:optconddisc}, we get that $\bar u \in U_{\alpha}^+$ solves \eqref{eq:Psigmaplus}, if and only if 
	\begin{equation*}
	J_{\sigma}'(\bar u) (u-\bar u) = \int_{\Omega} \varphi_{\bar u, \sigma}(0) \, d (u - \bar u) \geq 0 \qquad \forall u \in U_{\alpha}^+,
	\end{equation*}
	which is equivalent to the condition \eqref{eq:intineqdisc}.
	
	Now let $\bar u (\bar \Omega) = \alpha$. If $\bar \lambda = \min_{x \in \bar{\Omega}} \varphi_{\bar u, \sigma}(x,0) >0$, then take $u=0 \in U_{\alpha}^+$ in \eqref{eq:intineqdisc} to see that in this case $\bar u = 0$ must hold. So we must have $\bar \lambda \leq 0$. Furthermore, we can equivalently write \eqref{eq:intineqdisc} as 
	\begin{equation*}
	\int_{\Omega} \varphi_{\bar u, \sigma}(x,0) \, d \bar u = \min_{u \in U_{\alpha}^+} \int_{\Omega} \varphi_{\bar u, \sigma} (x,0) \, d u.
	\end{equation*}
	Take $x_0 \in \bar \Omega$, such that $\varphi_{\bar u,\sigma}(x_0,0) = \bar \lambda$. Then $u = \alpha \delta_{x_0}$ achieves the minimum in the equation above and we get \eqref{eq:lambdabardisc}. The other direction of the equivalence is obvious and completes the proof of part 1.
	
	In order to prove part 2, we look at two cases. First, let $\bar \lambda = 0$. By definition of $\bar \lambda$ this implies that $\varphi_{\bar u, \sigma} \geq 0 $ for all $x \in \bar \Omega$. So with \eqref{eq:lambdabardisc} we get that $\bar u$ has support, where $\varphi_{\bar u, \sigma}(x,0)= 0 = \bar \lambda$, in order for the integral to be zero. 
	
	The second case is $\bar \lambda <0$. Define $\psi(x):= - \min \left\{ \varphi_{\bar u, \sigma}(x,0), 0 \right\}$, then it holds $0 \leq \psi(x) \leq - \bar \lambda$ by definition of $\psi(x)$ and $\bar \lambda$. Furthermore $\| \psi \|_{\mathcal{C}(\bar \Omega)} = - \bar \lambda$. With \eqref{eq:intineqdisc} and $\psi(x) \geq - \varphi_{\bar u, \sigma}(x,0)$, we find
	\begin{equation*}
	\int_{\Omega}\psi(x) \, d \bar u \geq - \int_{\Omega} \varphi_{\bar u,\sigma}(x,0) \, d \bar u \geq - \int_{\Omega} \varphi_{\bar u, \sigma}(x,0) \, d u \qquad \forall u \in U_{\alpha}^+.
	\end{equation*}
	Especially for $u = \alpha \delta_{x_0}$, we have
	\begin{equation*}
	\int_{\Omega}\psi(x) \, d \bar u \geq - \int_{\Omega} \varphi_{\bar u, \sigma}(x,0) \, d (\alpha \delta_{x_0}) = - \alpha \bar \lambda = \| \bar u \|_{\mathcal{M}(\bar \Omega)} \| \psi\|_{\mathcal{C}(\bar \Omega)}.
	\end{equation*}
	Furthermore, we obviously have 
	\begin{equation*}
	\int_{\Omega}\psi(x) \, d \bar u \leq  \| \bar u \|_{\mathcal{M}(\bar \Omega)} \| \psi\|_{\mathcal{C}(\bar \Omega)},
	\end{equation*}
	so we can deduce equality and by \cite[Lemma 3.4]{CasasClasonKunisch} we then get \eqref{eq:suppposdisc}.
	
	The converse implication can be seen, since for a positive control $\bar u \in U_{\alpha}^+$ with $\bar u (\bar \Omega) = \alpha$, we can follow \eqref{eq:intineqdisc} from the condition \eqref{eq:suppposdisc}.
	
\end{proof}

\section{Numerical results}
\label{sec:Num}

For the implementation we consider $b \equiv 0$ and $c \equiv 0 $ in \eqref{eq:operatorA}.

We will consider the case of positive sources first, since the implementation is straightforward, while the general case requires to handle absolute values in the constraints. 

\subsection{Positive sources (problem \eqref{eq:Psigmaplus})}  
\label{subsec:pos}

We recall the discrete state equation \eqref{eq:dse}, which reduces to the following form, since $b \equiv 0$ and $c \equiv 0 $, with $z_h \in Y_h$: 
\begin{equation*}
\begin{cases}
\left( y_{k,h}-y_{k-1,h} , z_h \right)_{L^2} + \tau_k \int_{\Omega}{\nabla y_{k,h} \nabla z_h \, dx}    =  \int_{I_k}\int_{\Omega} f\, z_h \, dx\, dt, \\
y_{0,h}=y_{0h},
\end{cases}
\end{equation*}
where $y_{0h} \in Y_h$, for given $u \in \mathcal{M}(\bar \Omega)$, is the unique element satisfying:
\begin{equation*}
(y_{0h},z_h) = \int_{\Omega}{z_h \, d u} \qquad \forall \, z_h \in Y_h.
\end{equation*}
We define the mass matrix $M_h := \left( \left( \phi_{j}, \phi_{k} \right)_{L^2}\right)_{j,k=1}^{N_h}$ and the stiffness matrix $A_h := \left( \int_{\bar{\Omega}} \nabla \phi_{j} \nabla \phi_{k} \right)_{j,k=1}^{N_h}$ corresponding to $Y_h$.

We also notice that the matrix $ \left(\phi_{j}, \delta_{x_k}\right)_{j,k=1}^{N_h}$ is the identity in $\mathbb{R}^{N_h \times N_h}$. We represent the discrete state equation by the following operator $L: \mathbb{R}^{N_{\sigma}} \rightarrow \mathbb{R}^{N_{\sigma}} $:
\begin{align}
\underbrace{
	\begin{pmatrix} & M_h &   & &0 \\
	&-M_h & M_h +\tau_1 A_h & &\\
	& &\ddots &\ddots & \\
	&0 & &-M_h &M_h +\tau_{N_{\tau}}A_h  
	\end{pmatrix} }_{=:L}
\begin{pmatrix}
y_{0,h} \\ y_{1,h} \\ \vdots \\ y_{N_{\tau},h} 
\end{pmatrix}  
= 
\begin{pmatrix}
u \\ \tau_1 M_h f_{1,h} \\ \vdots \\ \tau_{N_{\tau}} M_h f_{N_{\tau},h} 
\end{pmatrix} .
\label{eq:statematrix}
\end{align}

We can now formulate the following finite-dimensional formulation of the discrete problem \eqref{eq:Psigmaplus}:
\begin{align}
\min_{u \in \mathbb{R}^{N_h}} J(u) = \frac{1}{2} &\left(y_{N_{\tau},h}(u) - y_d \right)^{\top} M_h \left(y_{N_{\tau},h}(u) - y_d \right),\label{eq:Palphadiscreteplus} \tag{$P^+_{h}$}\\
\textrm{s.t.} \qquad \sum_{i=1}^{N_h} u_i - \alpha &\leq 0, \notag\\
- u_i &\leq 0, \qquad \forall \, i \in \{1,\ldots,N_h\}. \notag 
\end{align}
The corresponding Lagrangian function $\mathcal{L}(u,\mu^{(1)},\mu^{(2)})$ with $\mu^{(1)} \in \mathbb{R}, \mu^{(2)}\in \mathbb{R}^{N_h}$ is defined by 
\begin{equation*}
\mathcal{L}(u,\mu^{(1)},\mu^{(2)}) := J(u) + \mu^{(1)} \left(\sum_{i=1}^{N_h} u_i - \alpha \right) - \sum_{i=1}^{N_h} \mu_i^{(2)} u_i .
\end{equation*}
All inequalities in \eqref{eq:Palphadiscreteplus} are strictly fulfilled for $u_i = \frac{\alpha}{N_h+1}$ for all $i \in \{1,\ldots,N_h\}$, thus an interior point of the feasible set exists, and the Slater condition is satisfied (see e.g. \cite[(1.132)]{HPUU}). Then the Karush-Kuhn-Tucker conditions (see e.g. \cite[(5.49)]{Boyd}) state that at the minimum $u$ the following conditions hold:
\begin{enumerate}
	\item $\partial_u \mathcal{L}(u,\mu^{(1)},\mu^{(2)}) = 0$,
	\item $\mu^{(1)} \left(\sum_{i=1}^{N_h} u_i - \alpha \right)  = 0 \; \wedge \; \mu^{(1)} \geq 0 \; \wedge \; \left(\sum_{i=1}^{N_h} u_i - \alpha \right) \leq 0$,
	\item $-\mu_i^{(2)} u_i = 0 \; \wedge \; \mu_i^{(2)} \geq 0 \; \wedge \; -u_i \leq 0 \quad \forall \, i \in \{1,\ldots,N_h\} $,
\end{enumerate}
where 2. and 3. can be equivalently reformulated with an arbitrary $\kappa>0$ by 

\begin{align*}
N^{(1)}(u,\mu^{(1)}) &:= \max \{ 0, \mu^{(1)} + \kappa \left(\textstyle\sum_{i=1}^{N_h} u_i - \alpha \right)  \} - \mu^{(1)} = 0, \\
N^{(2)}(u,\mu^{(2)}) & := \max \{ 0, \mu^{(2)} - \kappa u  \} - \mu^{(2)}  = 0 .
\end{align*}

We define 
\begin{equation*}
F(u,\mu^{(1)},\mu^{(2)}) := \begin{pmatrix}
& \partial_u \mathcal{L}(u,\mu^{(1)},\mu^{(2)}) & N^{(1)}(u,\mu^{(1)}) & N^{(2)}(u,\mu^{(2)}) 
\end{pmatrix}^{\top},
\end{equation*}
and apply the semismooth Newton method to solve $F(u,\mu^{(1)},\mu^{(2)}) = 0$. We have
\begin{equation*}
\partial_u \mathcal{L}(u,\mu^{(1)},\mu^{(2)}) = \partial_u J(u) + \mu^{(1)} \mathbb{1}_{N_h} - \mu^{(2)}.
\end{equation*} 
When setting up the matrix $DF = DF(u,\mu^{(1)},\mu^{(2)})$, we always choose $\partial_x (\max \{0,g(x)\}) = \partial_x g(x)$ if $g(x) =0$. This delivers 
\begin{equation*}
DF := \begin{pmatrix}
&\partial^2_u \mathcal{L}(u,\mu^{(1)},\mu^{(2)}) & \partial_{\mu^{(1)}} (\partial_u \mathcal{L}(u,\mu^{(1)},\mu^{(2)})) & \partial_{\mu^{(2)}}(\partial_u \mathcal{L}(u,\mu^{(1)},\mu^{(2)}))\\
& \partial_u N^{(1)}(u,\mu^{(1)}) & \partial_{\mu^{(1)}} N^{(1)}(u,\mu^{(1)}) & 0\\
& \partial_{u}  N^{(2)}(u,\mu^{(2)})  & 0 & \partial_{\mu^{(2)}}  N^{(2)}(u,\mu^{(2)}) 
\end{pmatrix},
\end{equation*}
with the entries: 
\begin{align*}
\partial^2_u \mathcal{L}(u,\mu^{(1)},\mu^{(2)}) &= \partial^2_u J(u),\\
\partial_{\mu^{(1)}} (\partial_u \mathcal{L}(u,\mu^{(1)},\mu^{(2)})) &= \mathbb{1}_{N_h},\\
\partial_{\mu^{(2)}}(\partial_u \mathcal{L}(u,\mu^{(1)},\mu^{(2)})) &= -\mathbb{1}_{N_h \times N_h},\\
\partial_u N^{(1)}(u,\mu^{(1)}) &= \begin{cases}
\kappa \, \mathbb{1}_{N_h}^{\top} , \quad &\mu^{(1)}+ \kappa \left(\sum_{i=1}^{N_h} u_i - \alpha \right) \geq 0, \\
0, \quad &\textrm{else},
\end{cases}\\ 
\partial_{\mu^{(1)}} N^{(1)}(u,\mu^{(1)}) &= \begin{cases}
0 , \quad &\mu^{(1)}+ \kappa \left(\sum_{i=1}^{N_h} u_i - \alpha \right) \geq 0, \\
-1, \quad &\textrm{else},
\end{cases}\\ 
\partial_{u_j}  N_i^{(2)}(u,\mu^{(2)}) &= \begin{cases}
- \kappa\,\delta_{ij} , \quad &\mu_i^{(2)} - \kappa u_i \geq 0,\\
0, &\textrm{else} ,
\end{cases}
\\ 
\partial_{\mu_j^{(2)}} N_i^{(2)}(u,\mu^{(2)}) &= \begin{cases}
0, \quad &\mu_i^{(2)} - \kappa u_i \geq 0,\\
-\delta_{ij}, &\textrm{else} .
\end{cases}
\end{align*}

\textbf{Numerical example}\\

Let $\Omega = [0,1]$, $T =1$ and $a = \frac{1}{100}$. We are working on an equidistant $20 \times 20$ grid for this example.

To generate a desired state $y_d$, we choose $u_{\operatorname{true}} = \delta_{0.5}$ and $f \equiv 0$ , solve the state equation on a very fine grid ($1000 \times 1000$) and take the evaluation of the result in $t = T$ on the current grid $\Omega_h$ as desired state $y_d$ (see Figure 1). 
Now we can insert this $y_d$ into our problem and solve for different values of $\alpha$. Knowing the true solution $u_{\operatorname{true}}$, we can compare our results to it. We also know $u_{\operatorname{true}}(\Omega) = 1$ and $\supp(u_{\operatorname{true}}) = \{0.5\}$. We always start the algorithm with the control being identically zero and terminate when the residual is below $10^{-15}$.

\newlength{\imgwidth}
\setlength{\imgwidth}{0.25\textwidth}
\begin{figure}[ht]
	\begin{center}
		\setlength{\tabcolsep}{1pt}
		\begin{tabular}{|c|c|c|}
			\hline
			\raisebox{-1\height}{ \includegraphics[width=\imgwidth]{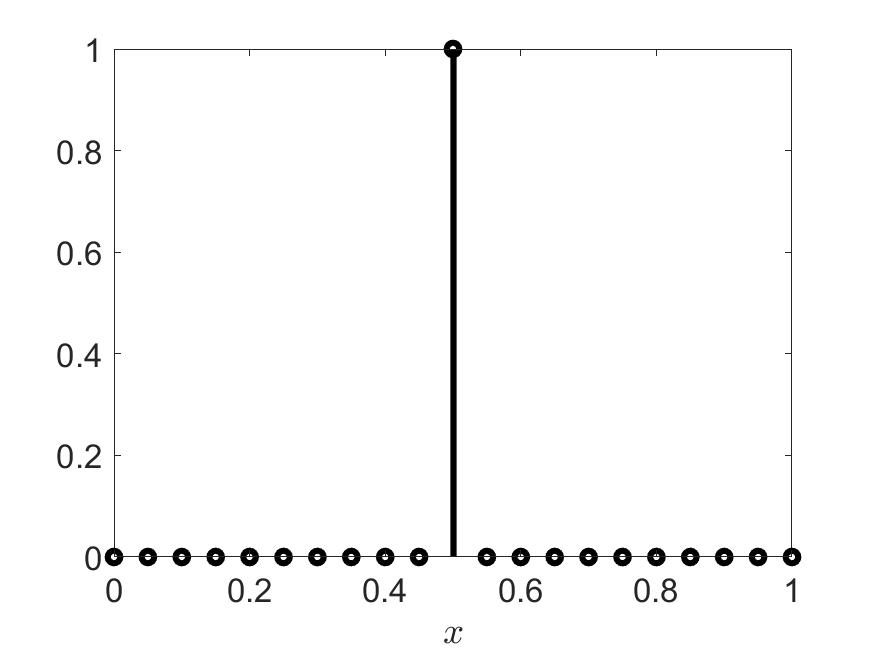}}
			&\raisebox{-1\height}{ \includegraphics[width=\imgwidth]{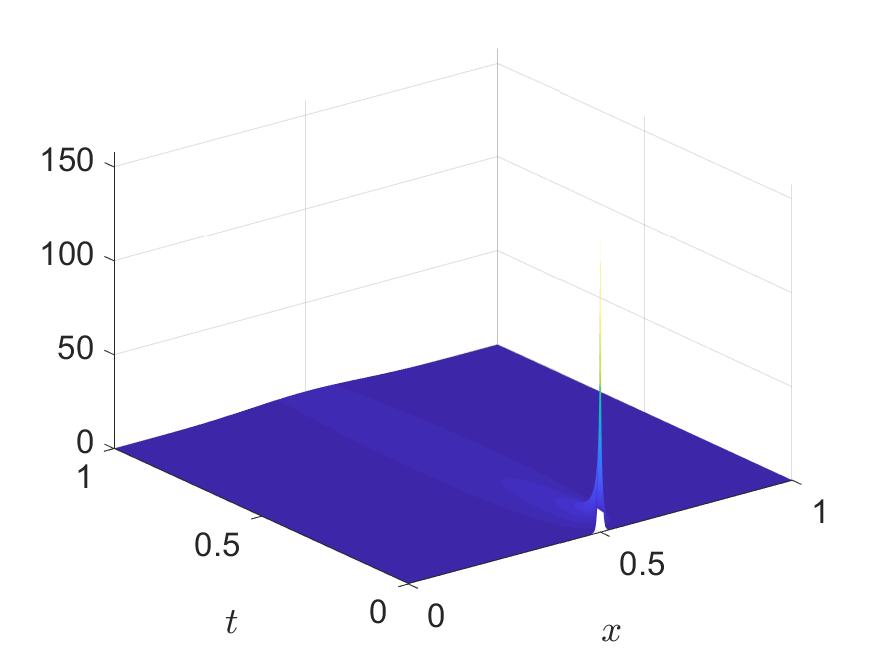}}
			&\raisebox{-1\height}{ \includegraphics[width=\imgwidth]{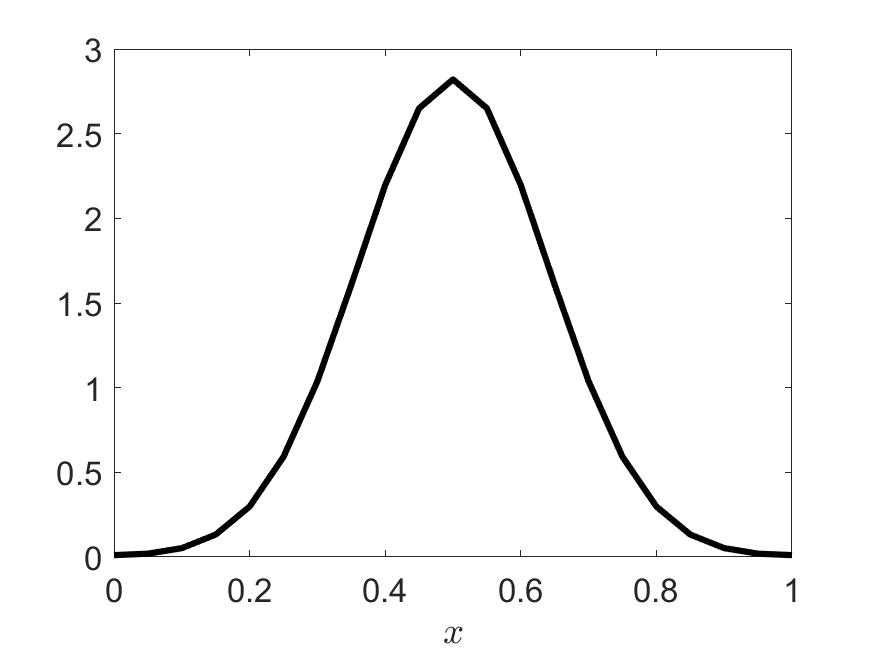}}
			\\
			\hline
		\end{tabular}
	\end{center}
	\caption{From left to right: true solution $u_{\operatorname{true}}$, associated true state $y_{\operatorname{true}}$ in $Q =[0,1] \times [0,1]$ and desired state $y_d = y_{\operatorname{true}}(T)$.}	
	\label{fig:1}
\end{figure}

The first case we investigate is $\alpha = 0.1$ (see Figure 2). This $\alpha$ is smaller than the total variation of the true control and we observe $\bar{u}(\bar{\Omega}) = \alpha$. Furthermore $\bar{\lambda} = \min_{x \in \bar{\Omega}} \bar{\varphi}(0) \approx -35.859$ and we can verify the optimality conditions \eqref{eq:lambdabardisc} and \eqref{eq:suppposdisc}, since 
\begin{equation*}
\int_{\Omega} \bar{\varphi}_h(x,0) \, d \bar{u} \approx -3.5859 \approx \alpha \bar{\lambda} ,
\end{equation*}
and $\supp(\bar{u}) = \{0.5\}$.

\begin{figure}[ht]
	\begin{center}
		\setlength{\tabcolsep}{1pt}
		\begin{tabular}{|c|c|c|c|}
			\hline	
			\raisebox{-1\height}{\includegraphics[width=\imgwidth]{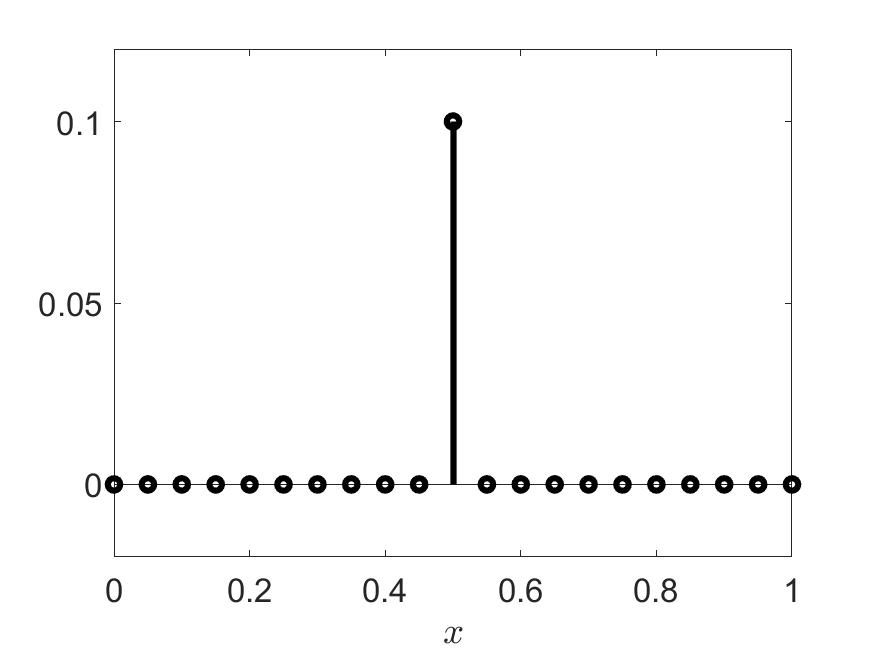}}
			&\raisebox{-1\height}{\includegraphics[width=\imgwidth]{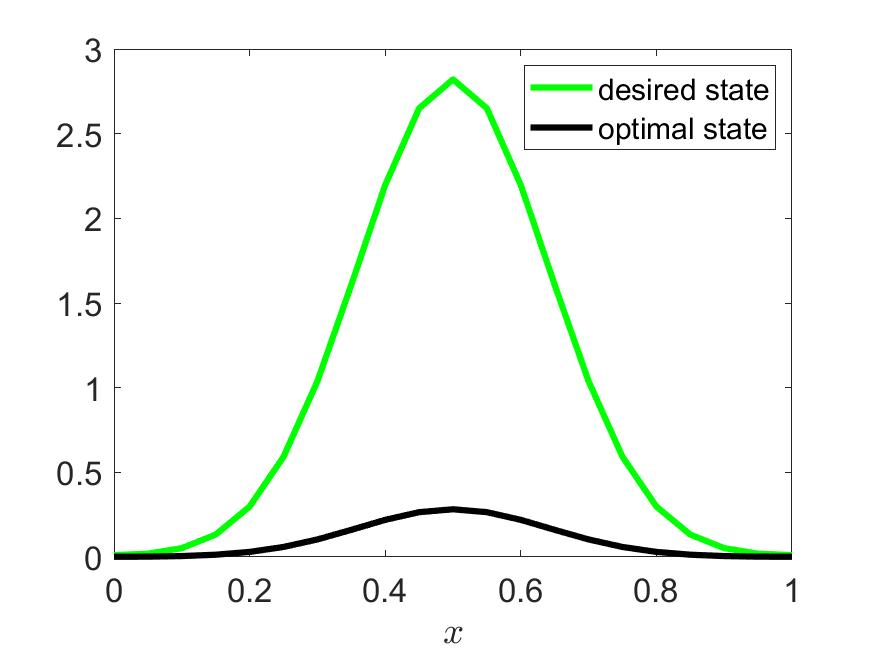}}
			&\raisebox{-1\height}{\includegraphics[width=\imgwidth]{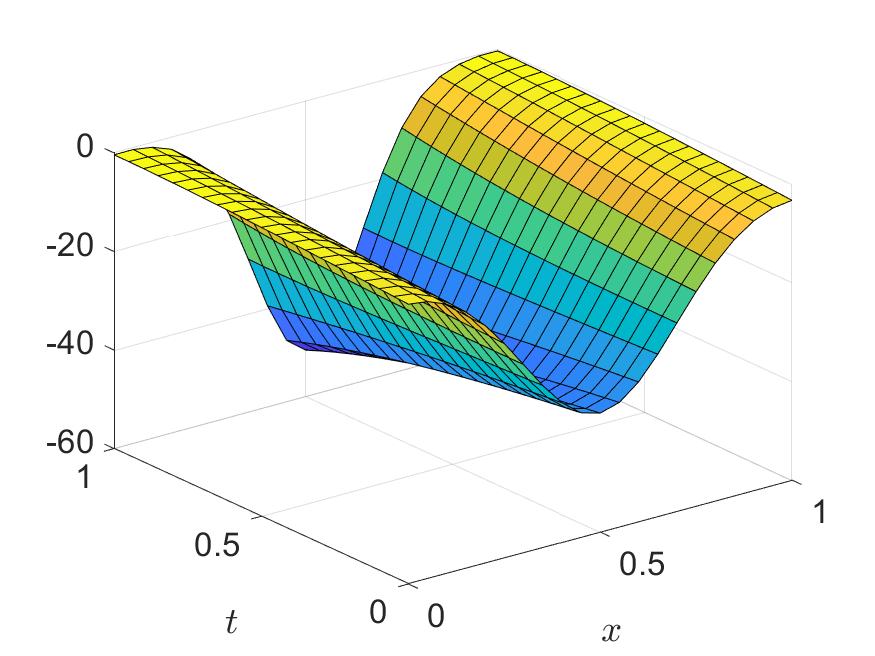}}
			&\raisebox{-1\height}{\includegraphics[width=\imgwidth]{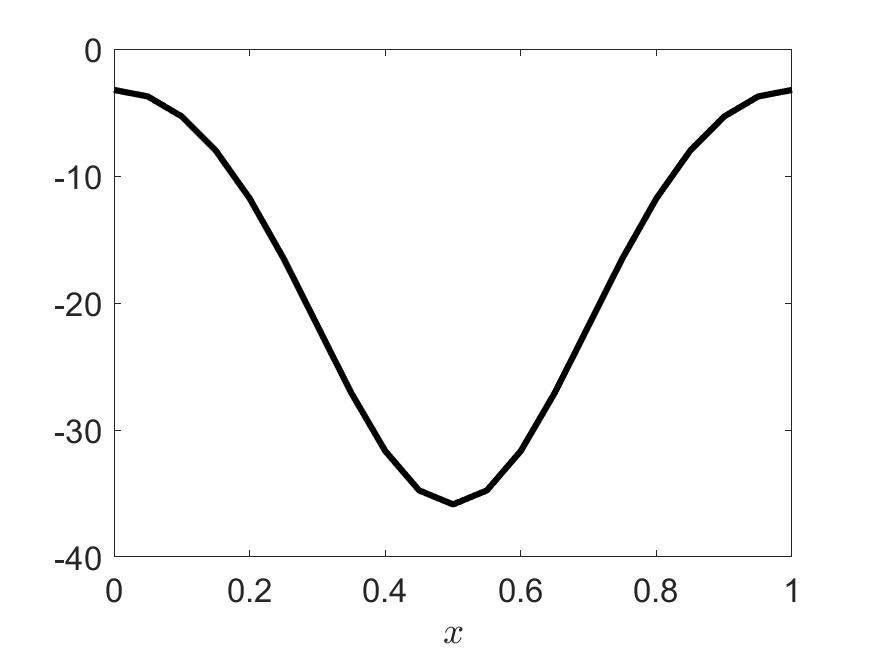}}
			\\
			\hline
		\end{tabular}
	\end{center}
	\caption{Solutions for $\alpha = 0.1$: from left to right: optimal control $\bar{u}$ (solved with the semismooth Newton method), associated optimal state $\bar{y}$, associated adjoint $\bar{\varphi}$ on the whole space-time domain $Q$, associated adjoint $\bar{\varphi}$ at $t=0$. Terminated after 16 Newton steps.}
\end{figure}

The second case we investigate is $\alpha = 1 = u_{\operatorname{true}}(\bar{\Omega})$ (see Figure 3). The computed optimal control in this case has a total variation of $\bar{u}(\bar{\Omega}) = 1 = \alpha$ and we can again verify the sparsity $\supp(\bar{u}) = \{0.5\}$. Furthermore $\bar{\lambda} = \min_{x \in \bar{\Omega}} \bar{\varphi}(0) \approx -0.0436$ and we can verify the optimality condition \eqref{eq:lambdabardisc}, since 
\begin{equation*}
\int_{\Omega} \bar{\varphi}_h(x,0) \, d \bar{u} \approx -0.0436 \approx \alpha \bar{\lambda} .
\end{equation*}

\begin{figure}[ht]
	\begin{center}
		\setlength{\tabcolsep}{1pt}
		\begin{tabular}{|c|c|c|c|}
			\hline	
			\raisebox{-1\height}{\includegraphics[width=\imgwidth]{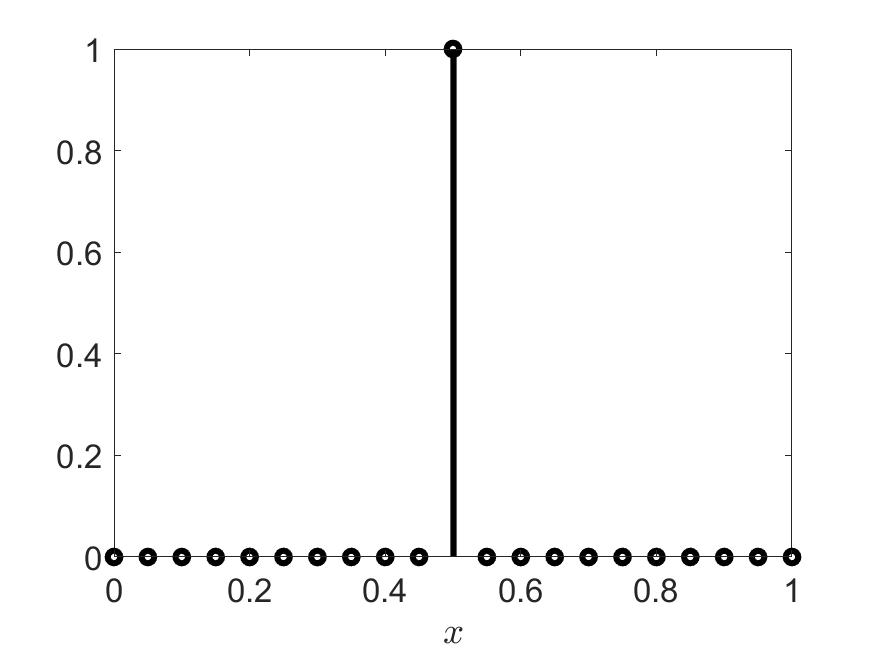}}
			&\raisebox{-1\height}{\includegraphics[width=\imgwidth]{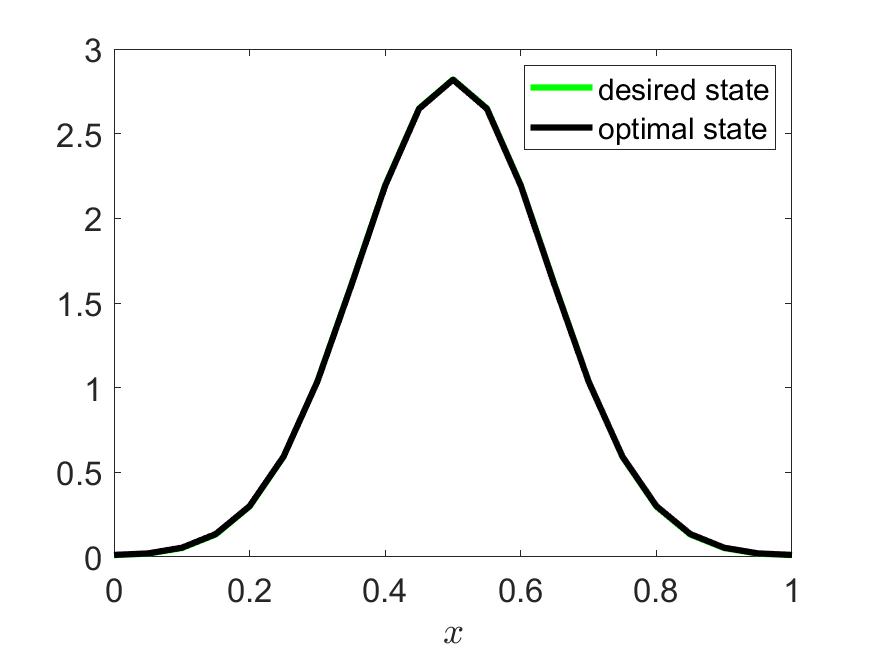}}
			&\raisebox{-1\height}{\includegraphics[width=\imgwidth]{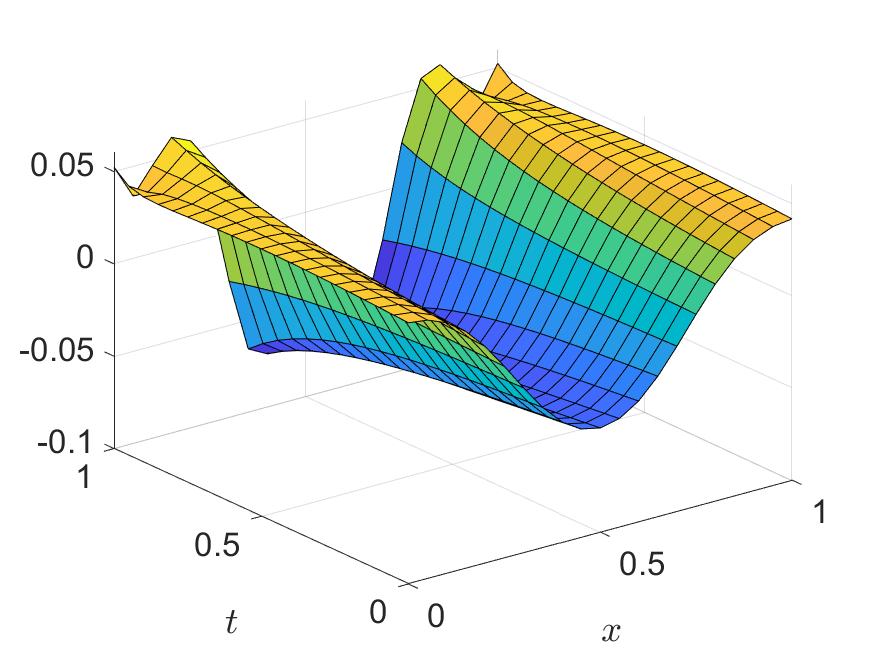}}
			&\raisebox{-1\height}{\includegraphics[width=\imgwidth]{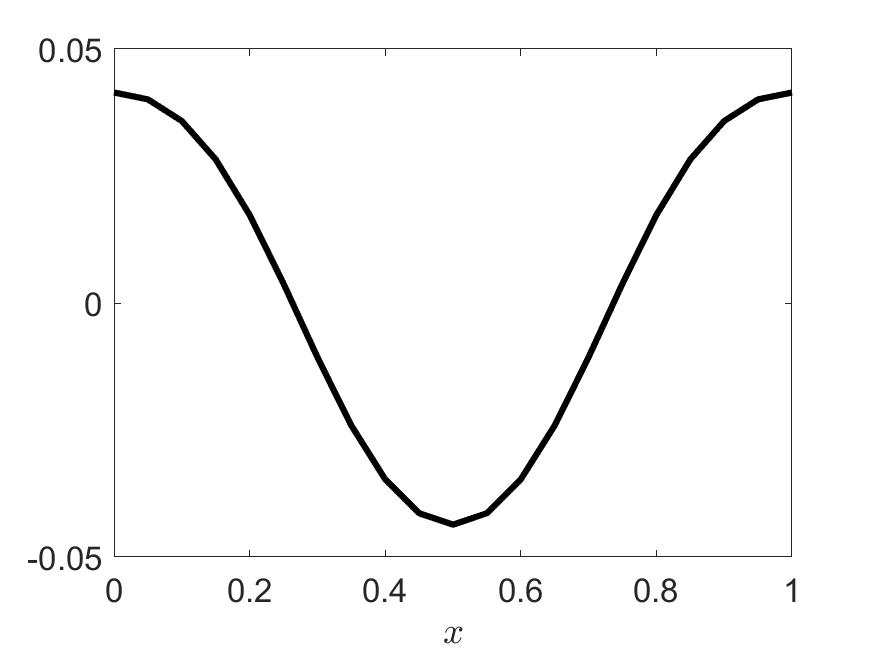}}
			\\
			\hline
		\end{tabular}
	\end{center}
	\caption{Solutions for $\alpha = 1$: from left to right: optimal control $\bar{u}$ (solved with the semismooth Newton method), associated optimal state $\bar{y}$, associated adjoint $\bar{\varphi}$ on the whole space-time domain $Q$, associated adjoint $\bar{\varphi}$ at $t=0$. Terminated after 15 Newton steps. }	
\end{figure}
For cases with $\alpha > u_{\operatorname{true}}(\bar{\Omega})$, we get similar results as in the case with $\alpha = 1$. In particular this means that we observe optimality conditions \eqref{eq:lambdabardisc} and \eqref{eq:suppposdisc}. Since we fixed $f \equiv 0$, we get $y_0(T) \equiv 0$ and therefore $y_d > y_0(T)$. Still, the properties that we found in the general case for $\bar{u}(\bar{\Omega}) < \alpha$: $\bar{y}(T) = y_d$ and $\varphi = 0 \in Q$ can not be observed (compare Figure 4 top). This is caused by the fact that the desired state $y_d$ can not be reached on the coarse grid, so $\bar{y}(T) = y_d$ is not possible. Solving the problem with a desired state that has been projected onto the coarse grid, thus is reachable, delivers the expected properties $\bar{y}(T) = y_d$ and $\varphi = 0 \in Q$ (see Figure 4 bottom).
For examples with $y_d \leq y_0(T)$ we can confirm Remark \ref{remark} and find the optimal solution $\bar{u} = 0$. 
\begin{figure}[ht]
	\begin{center}
		\setlength{\tabcolsep}{1pt}
		\begin{tabular}{|c|c|c|c|}
			\hline	
			\raisebox{-1\height}{\includegraphics[width=\imgwidth]{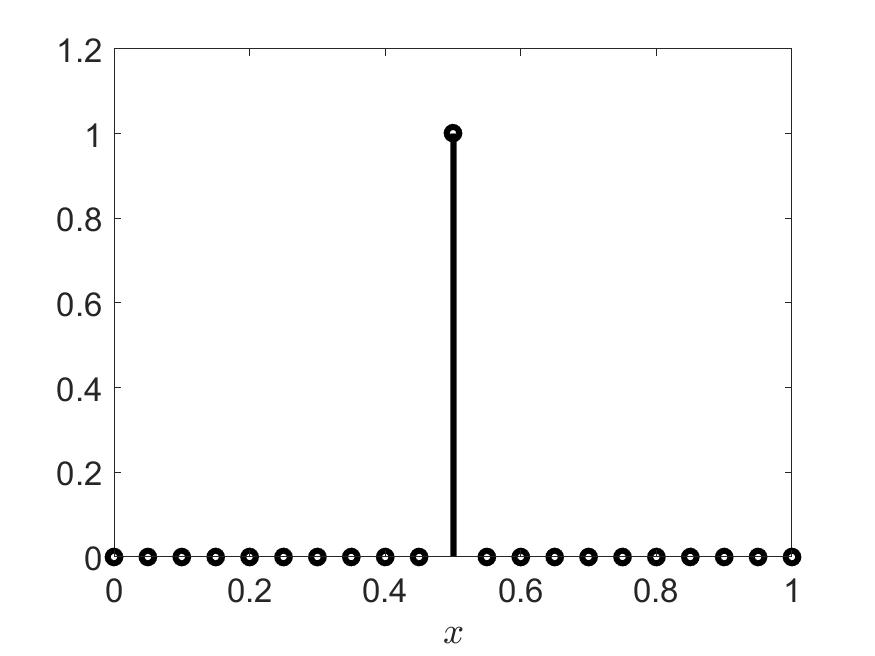}}
			&\raisebox{-1\height}{\includegraphics[width=\imgwidth]{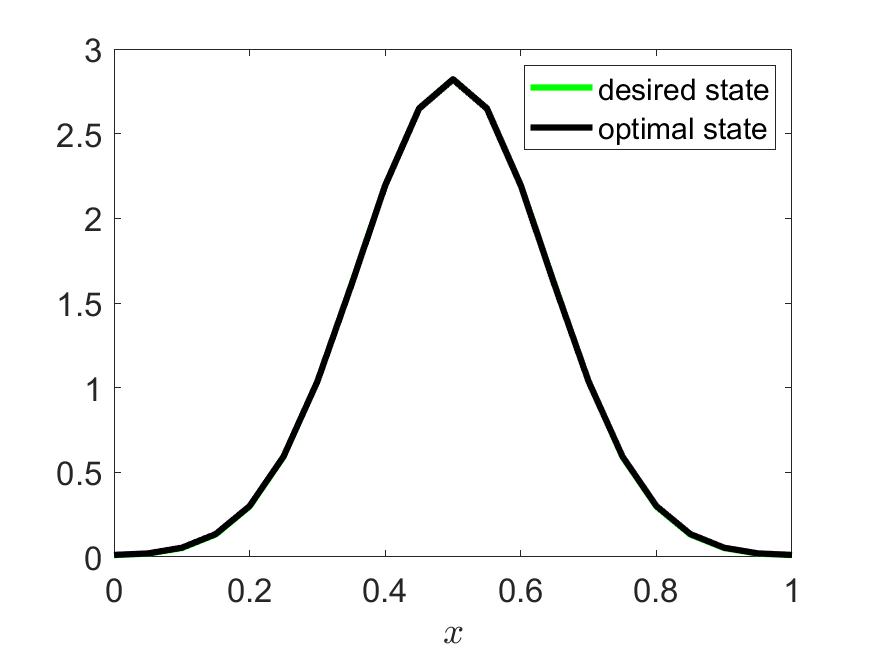}}
			&\raisebox{-1\height}{\includegraphics[width=\imgwidth]{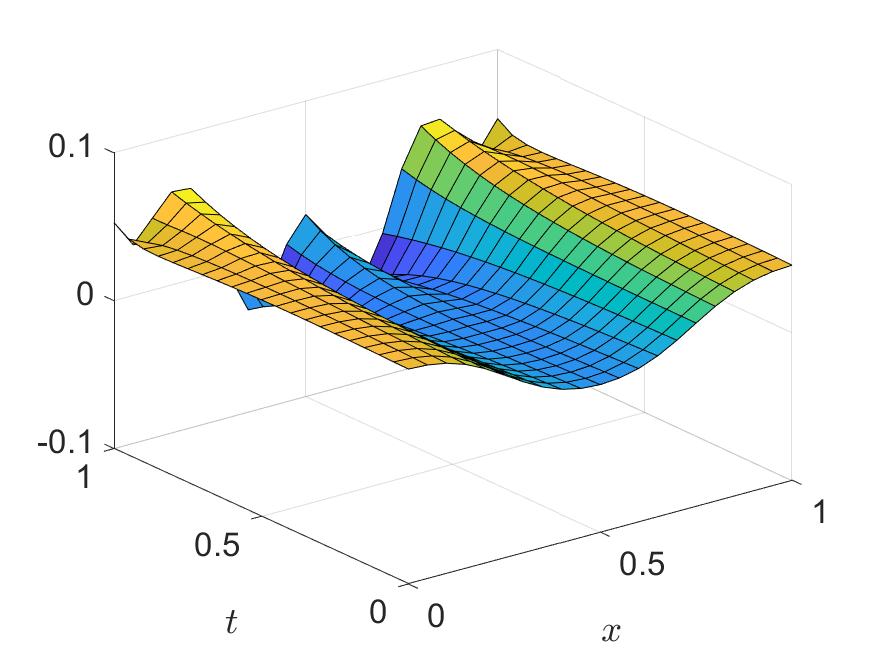}}
			&\raisebox{-1\height}{\includegraphics[width=\imgwidth]{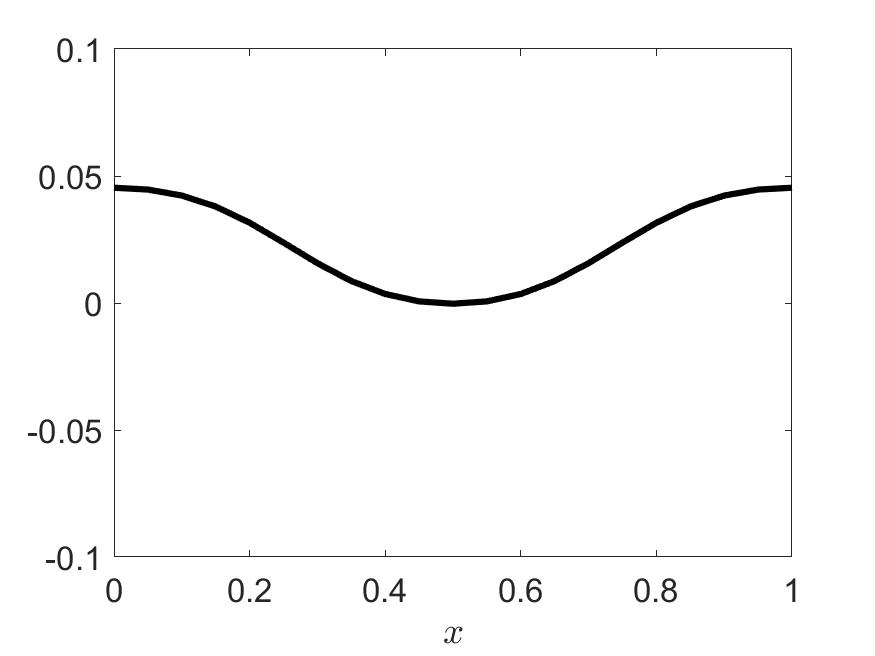}}
			\\
			\hline
			\raisebox{-1\height}{\includegraphics[width=\imgwidth]{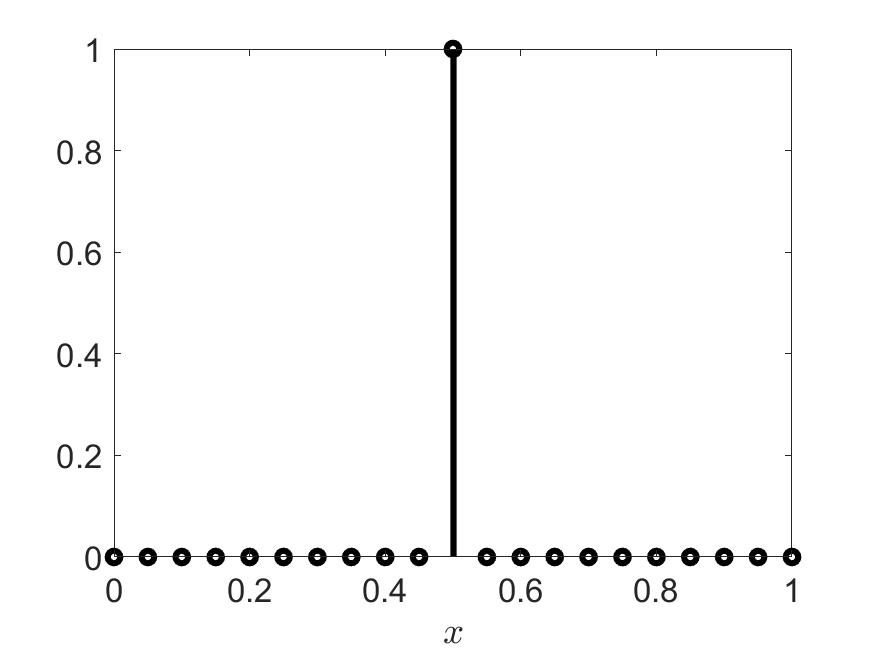}}
			&\raisebox{-1\height}{\includegraphics[width=\imgwidth]{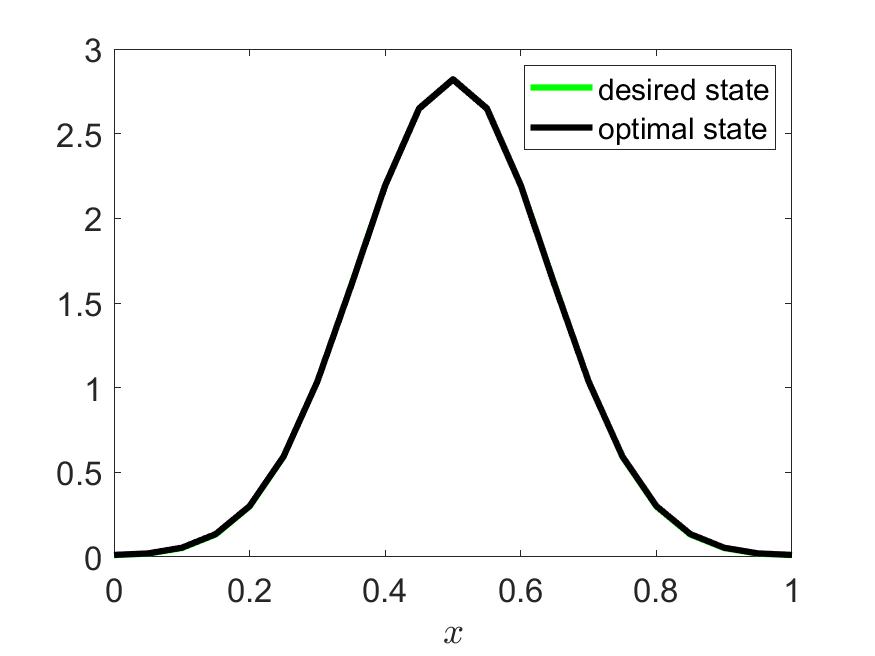}}
			&\raisebox{-1\height}{\includegraphics[width=\imgwidth]{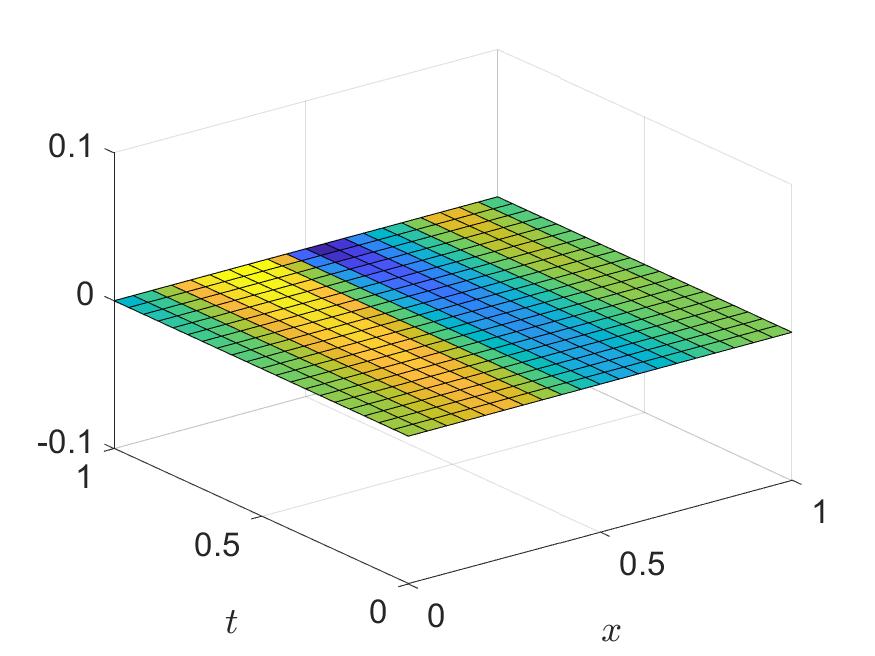}}
			&\raisebox{-1\height}{\includegraphics[width=\imgwidth]{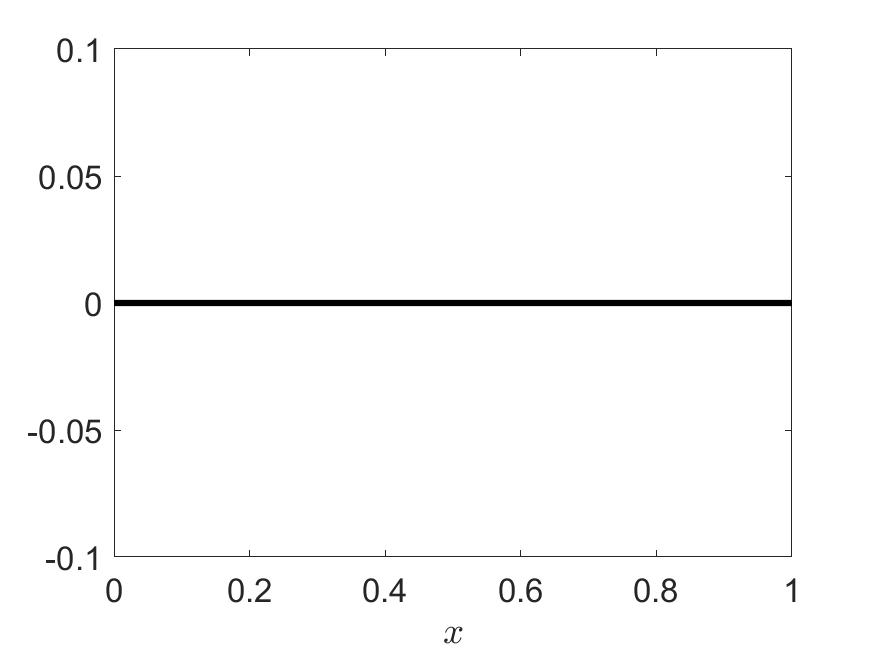}}
			\\
			\hline
		\end{tabular}
	\end{center}
	\caption{Solutions for $\alpha = 2$ with original desired state (top) and reachable desired state (bottom): from left to right: optimal control $\bar{u}$ (solved with the semismooth Newton method), associated optimal state $\bar{y}$, associated adjoint $\bar{\varphi}$ on the whole space-time domain $Q$, associated adjoint $\bar{\varphi}$ at $t=0$. Terminated after 17 and 27 Newton steps, respectively. }	
\end{figure}

\subsection{The general case (problem \eqref{eq:Psigma})}
\label{subsec:gen}
Here, the source does not need to be positive.

In the discrete problem we will decompose the control $u \in U_h$ into its positive and negative part, such that 
\begin{equation*}
u = u^+ - u^-, \qquad u^+ \geq 0, \quad u^- \geq 0.
\end{equation*}

We have the following finite-dimensional formulation of the discrete problem \eqref{eq:Psigma}:
\begin{align}
\min_{u^+,u^- \in \mathbb{R}^{N_h}} J(u^+,u^-) = \frac{1}{2} &\left(y_{N_{\tau},h}(u^+,u^-) - y_d \right)^{\top} M_h \left(y_{N_{\tau},h}(u^+,u^-) - y_d \right), \label{eq:Palphadiscrete} \tag{$P_{h}$}\\
\textrm{s.t.} \qquad \textstyle\sum_{i=1}^{N_h} |u^+_i - u^-_i| - \alpha &\leq 0, \notag \\
-u^+_i &\leq 0 \quad \forall \, i, \notag \\
-u^-_i &\leq 0 \quad \forall \, i, \notag 
\end{align}
where $y_{N_{\tau},h}(u^+,u^-)$ corresponds to solving \eqref{eq:statematrix} with $u = u^+ - u^-$ inserted into the right hand side of the equation. In order to allow taking second derivatives of the Lagrangian, we want to equivalently reformulate the absolute value in the first constraint. This can be done by adding the following constraint in our discrete problem:
\begin{equation}
u^+_i u^-_i = 0 \quad \forall \, i \label{eq:switch}.
\end{equation}
and consequently the first constraint becomes
\begin{equation*}
\left( \textstyle\sum_{i=1}^{N_h} u^+_i + u^-_i \right) - \alpha \leq 0.
\end{equation*} 
However, in the case $u^+_i = u^-_i = 0$, the matrix in the Newton step will be singular. Since we want to handle sparse problems, this case will very likely occur, so we need to find a way to overcome this difficulty. Instead of adding an additional constraint, we could also add a penalty term that enforces $u^+_i u^-_i = 0 \quad \forall \, i$ and consider the problem
\begin{align}
\min_{u^+,u^- \in \mathbb{R}^{N_h}} J(u^+,u^-) + \gamma (u^+)&^{\top}u^-, \label{eq:Palphadiscretepenalty} \tag{$P_{h, \gamma}$}\\
\textrm{s.t.} \qquad \textstyle\sum_{i=1}^{N_h} u^+_i + u^-_i - \alpha &\leq 0, \notag \\
-u^+_i &\leq 0\, \quad \forall \, i, \notag \\
-u^-_i &\leq 0\, \quad \forall \, i. \notag 
\end{align}
For $\gamma$ large enough the solutions of \eqref{eq:Palphadiscretepenalty} and \eqref{eq:Palphadiscrete} will coincide. In \cite[Theorem 4.6]{han1979exact} it is specified that $\gamma$ should be larger than the largest absolute value of the Karush-Kuhn-Tucker multipliers corresponding to the equality constraints \eqref{eq:switch}, which are replaced.

We have the corresponding Lagrangian with $\mu^{(1)} \in \mathbb{R}, \mu^{(2)}, \mu^{(3)} \in \mathbb{R}^{N_h}$:
\begin{align*}
\mathcal{L}(u^+,u^-,\mu^{(1)},\mu^{(2)},\mu^{(3)}) := &J(u^+,u^-) + \gamma (u^+)^{\top} u^- + \mu^{(1)} \left(\textstyle\sum_{i=1}^{N_h} u^+_i - u^-_i - \alpha \right) \\
&- \sum_{i=1}^{N_h} \mu_i^{(2)} u_i^+ - \sum_{i=1}^{N_h} \mu_i^{(3)} u_i^-.
\end{align*}

All inequalities in \eqref{eq:Palphadiscretepenalty} are strictly fulfilled for $u_i^+ = u_i^- = \frac{\alpha}{2 (N_h +1)}$ for all $i \in \{ 1, \ldots, N_h\}$, so the Slater condition is satisfied (see e.g. \cite[(1.132)]{HPUU}). By Karush-Kuhn-Tucker conditions (see e.g. \cite[(5.49)]{Boyd}) the following conditions in the minimum $(u^+,u^-)$ must be fulfilled, where we directly reformulate the inequality conditions with an arbitrary $\kappa >0 $ as in the case with positive measures.

\begin{enumerate}
	\item $\partial_{u^+} \mathcal{L}(u^+,u^-,\mu^{(1)},\mu^{(2)},\mu^{(3)}) = 0$,
	\item $\partial_{u^-} \mathcal{L}(u^+,u^-,\mu^{(1)},\mu^{(2)},\mu^{(3)}) = 0$,
	\item $N^{(1)}(u^+,u^-,\mu^{(1)}) = \max \{ 0, \mu^{(1)} + \kappa \left(\sum_{i=1}^{N_h} u^+_i - u^-_i - \alpha \right)  \} - \mu^{(1)} = 0$,
	\item $N^{(2)}(u^+,\mu^{(2)}) = \max \{ 0, \mu^{(2)} - \kappa u^+  \} - \mu^{(2)}  = 0$,
	\item $N^{(3)}(u^-,\mu^{(3)}) = \max \{ 0, \mu^{(3)} - \kappa u^-  \} - \mu^{(3)}  = 0$.
\end{enumerate}
We then apply the semismooth Newton method to solve
\begin{equation*}
F(u^+,u^-,\mu^{(1)},\mu^{(2)},\mu^{(3)}) := \begin{pmatrix}
\partial_{u^+} \mathcal{L}(u^+,u^-,\mu^{(1)},\mu^{(2)},\mu^{(3)}) \\
\partial_{u^-} \mathcal{L}(u^+,u^-,\mu^{(1)},\mu^{(2)},\mu^{(3)}) \\
N^{(1)}(u^+,u^-,\mu^{(1)})\\
N^{(2)}(u^+,\mu^{(2)}) \\
N^{(3)}(u^-,\mu^{(3)})
\end{pmatrix}
= 0.
\end{equation*} 
We have
\begin{align*}
\partial_{u^+} \mathcal{L}(u^+,u^-,\mu^{(1)},\mu^{(2)},\mu^{(3)}) &= \partial_{u^+} J(u^+,u^-) + \gamma u^- + \mu^{(1)} \mathbb{1}_{N_h} - \mu^{(2)} ,\\
\partial_{u^-} \mathcal{L}(u^+,u^-,\mu^{(1)},\mu^{(2)},\mu^{(3)}) &= \partial_{u^-} J(u^+,u^-) + \gamma u^+ + \mu^{(1)} \mathbb{1}_{N_h} - \mu^{(3)} .
\end{align*}

When setting up the matrix $DF$, we always make the choice $\partial_x (\max \{0,g(x)\}) = \partial_x g(x)$ if $g(x) =0$. This delivers (in short notation):
%
\begin{equation*}
DF := \begin{pmatrix}
&\partial^2_{u^+} \mathcal{L} &\partial_{u^-} \partial_{u^+}\mathcal{L} & \partial_{\mu^{(1)}} \partial_{u^+} \mathcal{L} & \partial_{\mu^{(2)}} \partial_{u^+} \mathcal{L} & 0 \\
&\partial_{u^+} \partial_{u^-} \mathcal{L} &\partial^2_{u^-} \mathcal{L} & \partial_{\mu^{(1)}} \partial_{u^-} \mathcal{L} & 0 & \partial_{\mu^{(3)}} \partial_{u^-} \mathcal{L}\\
& \partial_{u^+} N^{(1)} & \partial_{u^-} N^{(1)} & \partial_{\mu^{(1)}} N^{(1)} & 0 & 0\\
& \partial_{u^+} N^{(2)} & 0 & 0 & \partial_{\mu^{(2)}} N^{(2)} & 0 \\
& 0 & \partial_{u^-} N^{(3)} & 0 & 0 & \partial_{\mu^{(3)}} N^{(3)}
\end{pmatrix},
\end{equation*}

with the entries 
\begin{align*}
\partial^2_{u^+} \mathcal{L} &= \partial^2_{u^+} J \notag ,\\
\partial_{u^-} \partial_{u^+}\mathcal{L} = \partial_{u^+} \partial_{u^-} \mathcal{L} &= \partial_{u^-} \partial_{u^+} J + \gamma \mathbb{1}_{N_h \times N_h} , \\ 
\partial_{\mu^{(1)}} \partial_{u^+} \mathcal{L} = \partial_{\mu^{(1)}} \partial_{u^-} \mathcal{L} &= \mathbb{1}_{N_h}  ,\\
\partial_{\mu^{(2)}} \partial_{u^+} \mathcal{L} = \partial_{\mu^{(3)}} \partial_{u^-} \mathcal{L} &= -\mathbb{1}_{N_h \times N_h } \notag,\\
\partial^2_{u^-} \mathcal{L} &= \partial^2_{u^-} J , \\ 
\partial_{u^+} N^{(1)} = \partial_{u^-} N^{(1)} &= \begin{cases}
\kappa \mathbb{1}_{N_h}^{\top} , \quad &\mu^{(1)}+ \kappa \left(\sum_{i=1}^{N_h} u^+_i + u^-_i - \alpha \right) \geq 0 ,\\
0, \quad &\textrm{else},
\end{cases}  \\ 
\partial_{\mu^{(1)}} N^{(1)} &= \begin{cases}
0 , \quad &\mu^{(1)}+ \kappa \left(\sum_{i=1}^{N_h} u^+_i + u^-_i - \alpha \right) \geq 0, \\
-1, \quad &\textrm{else},
\end{cases} \\
\partial_{u^+_j} N_i^{(2)} &= \begin{cases}
-\kappa \delta_{ij}, \quad &\mu_i^{(2)} - \kappa u^+_i \geq 0,\\
0, &\textrm{else} ,
\end{cases} \\
\partial_{\mu_j^{(2)}} N_i^{(2)} &= \begin{cases}
0, \quad &\mu_i^{(2)} - \kappa u^+_i \geq 0,\\
-\delta_{ij}, &\textrm{else} ,
\end{cases}\\
\partial_{u_j^-} N_i^{(3)} &= \begin{cases}
- \kappa \delta_{ij}, \quad &\mu_i^{(3)} - \kappa u^-_i \geq 0,\\
0, &\textrm{else} ,
\end{cases}\\ 
\partial_{\mu_j^{(3)}} N_i^{(3)} &= \begin{cases}
0, \quad &\mu_i^{(3)} - \kappa u^-_i \geq 0,\\
-\delta_{ij}, &\textrm{else} .
\end{cases}
\end{align*}  
\newpage
\textbf{Numerical example}\\

Let $\Omega = [0,1]$, $T =1$ and $a = \frac{1}{100}$. We are working on a $20 \times 20$ grid for this example. Positive parts of the measure are displayed by black circles and negative parts by red diamonds.

We always start the algorithm with the control being identically zero and terminate when the residual is below $10^{-15}$.

First example like described in Section \ref{subsec:pos}, compare Figure \ref{fig:1}. We found the following values to be suitable: The penalty parameter $\gamma = 70$ in \ref{eq:Palphadiscretepenalty} and the multiplier $\kappa =2$ to reformulate the KKT-conditions.

The first case we investigate is $\alpha = 0.1$ (see Figure 5 top). This $\alpha$ is smaller than the total variation of the true control and we observe $\bar{u}^+(\bar{\Omega}) = \alpha, \bar{u}^-(\bar{\Omega}) = 0$. 
The second case we investigate is $\alpha = 1$ (see Figure 5 bottom). This $\alpha$ is equal to the total variation of the true control and we observe $\bar{u}^+(\bar{\Omega}) = \alpha, \bar{u}^-(\bar{\Omega}) = 1.8635 \cdot 10^{-20}$. These results are almost identical to the results in Section \ref{subsec:pos}, where only positive measures were allowed (compare Figure 2 and 3).

\begin{figure}[t]
	\begin{center}
		\setlength{\tabcolsep}{1pt}
		\begin{tabular}{|c|c|c|c|}
			\hline	
			\raisebox{-1\height}{\includegraphics[width=\imgwidth]{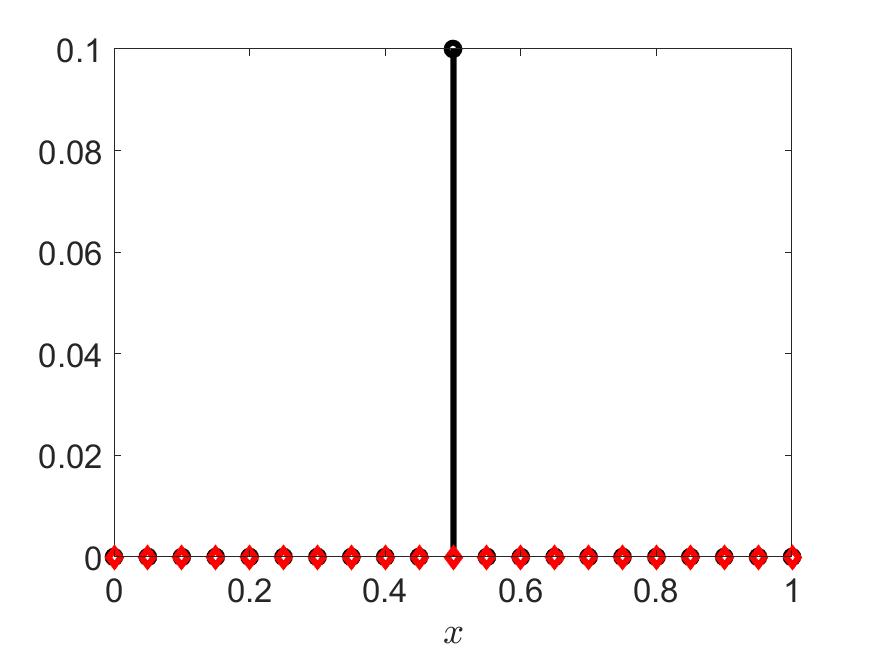}}
			&\raisebox{-1\height}{\includegraphics[width=\imgwidth]{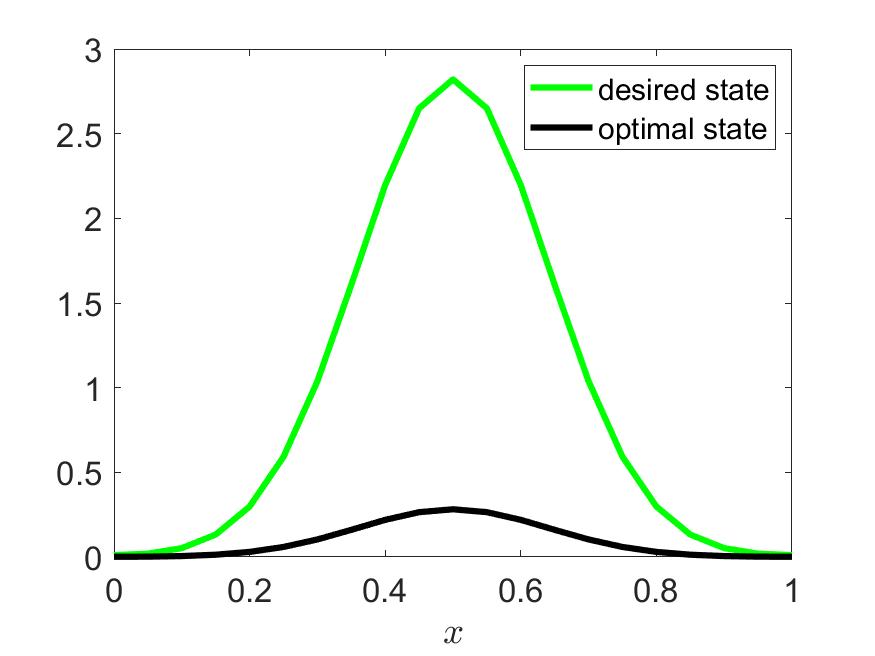}}
			&\raisebox{-1\height}{\includegraphics[width=\imgwidth]{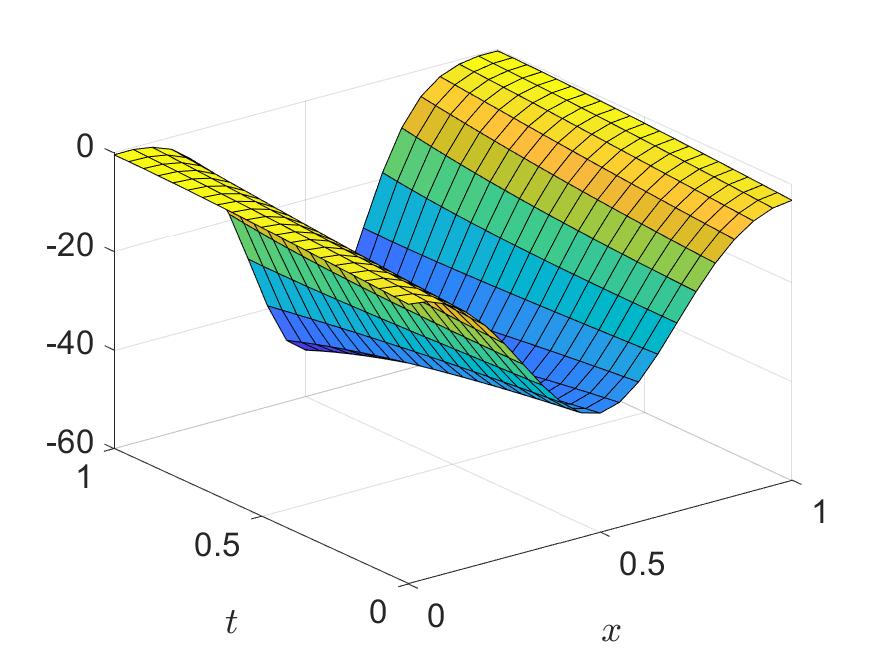}}
			&\raisebox{-1\height}{\includegraphics[width=\imgwidth]{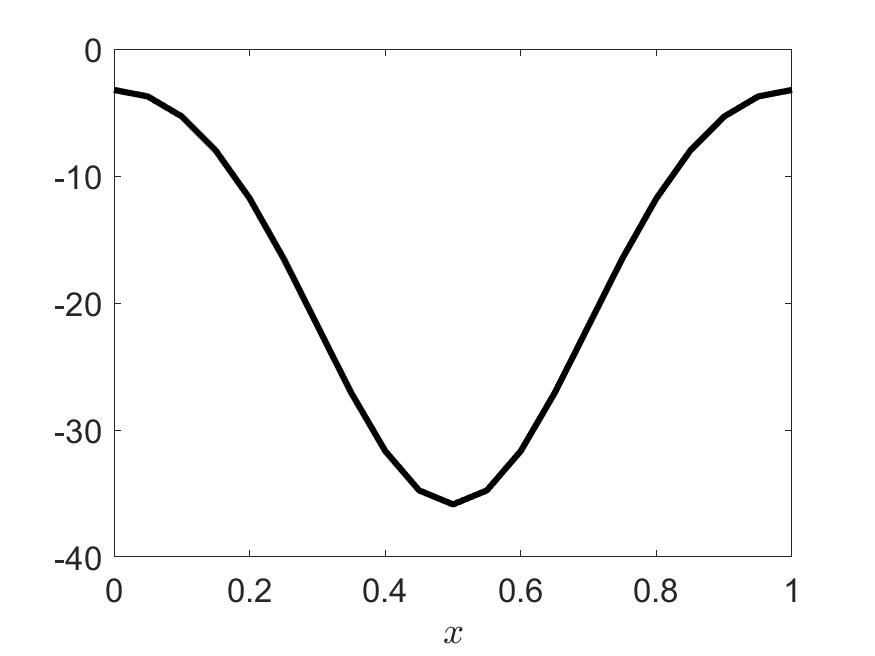}}
			\\
			\hline
			\raisebox{-1\height}{\includegraphics[width=\imgwidth]{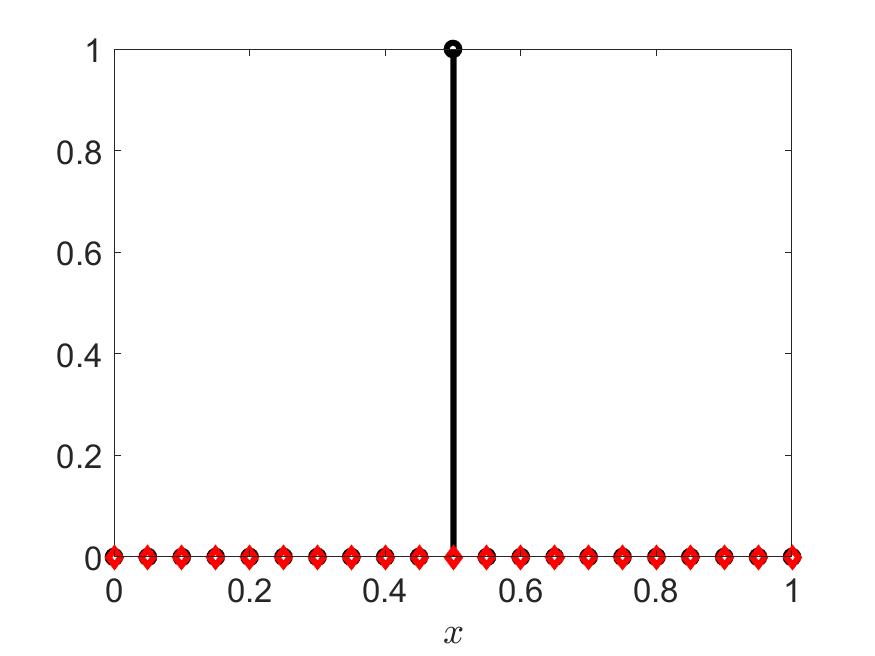}}
			&\raisebox{-1\height}{\includegraphics[width=\imgwidth]{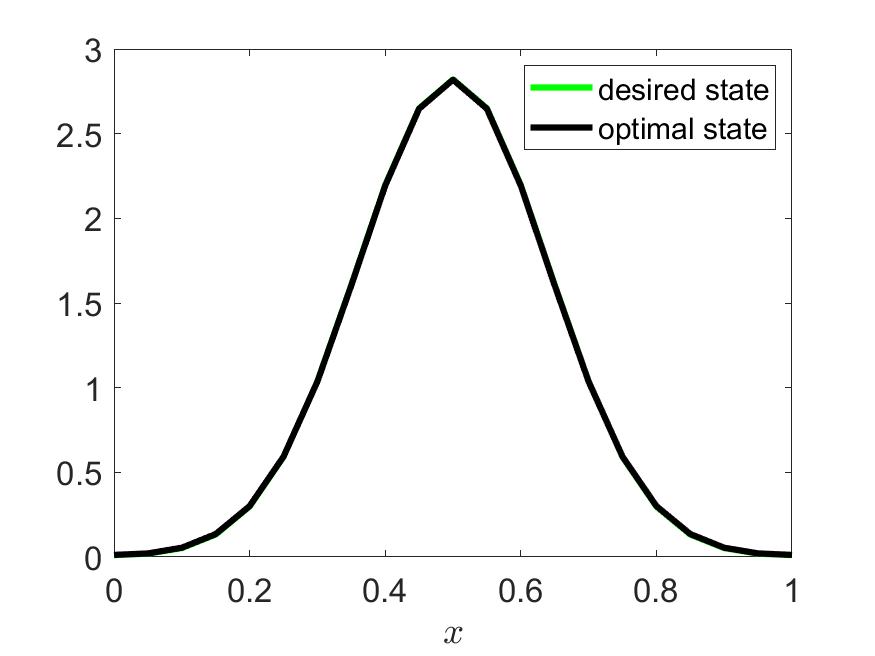}}
			&\raisebox{-1\height}{\includegraphics[width=\imgwidth]{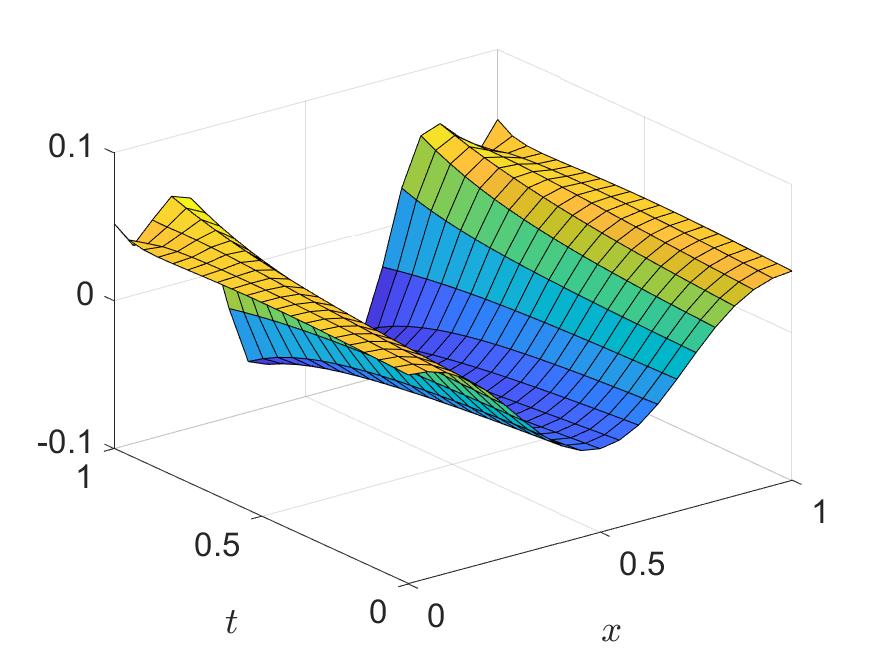}}
			&\raisebox{-1\height}{\includegraphics[width=\imgwidth]{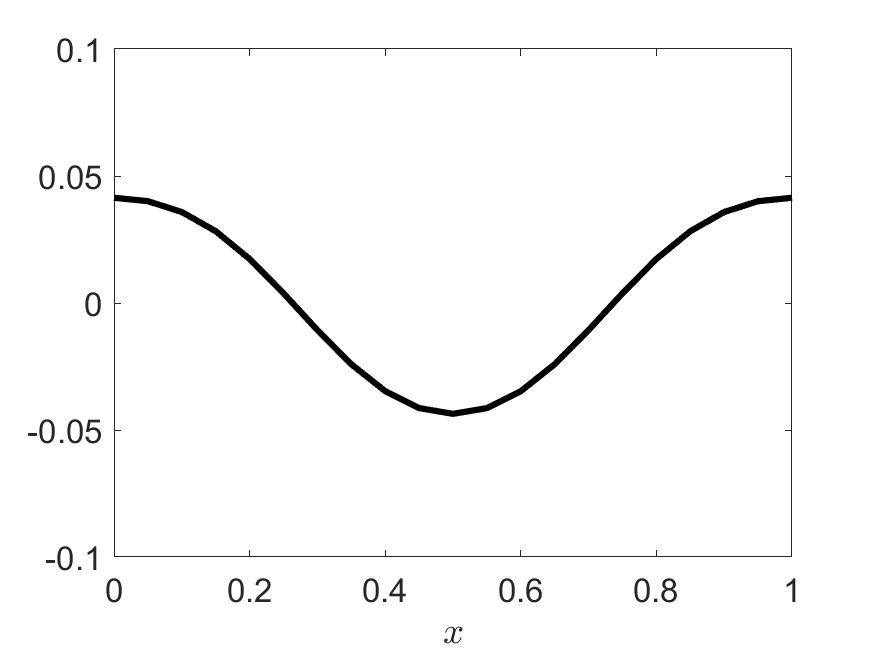}}
			\\
			\hline
		\end{tabular}
	\end{center}
	\caption{Solutions for $\alpha = 0.1$ (top) and $\alpha = 1$ (bottom): from left to right: optimal control $\bar{u}= \bar{u}^+ - \bar{u}^-$ (solved with the semismooth Newton method), associated optimal state $\bar{y}$, associated adjoint $\bar{\varphi}$ on the whole space-time domain $Q$, associated adjoint $\bar{\varphi}$ at $t=0$. Terminated after 11 and 64 Newton steps, respectively.}
\end{figure}


The third case we investigate is $\alpha = 2$ (see Figure 6). This $\alpha$ is bigger than the total variation of the true control and we observe $\bar{u}^+(\bar{\Omega}) = 1.5, \bar{u}^-(\bar{\Omega}) = 0.5 $. Furthermore $\bar y (T) \approx y_d$ (with an error of size $10^{-8}$) and $\bar \varphi \approx 0 \in Q$. Since we allow positive and negative coefficients, the desired state can be reached on the coarse grid - different to the case of only positive sources, but as a payoff the sparsity of the optimal control is lost.  As required, the complementarity condition has been fulfilled, i.e. $u_i^+ u_i^- = 0$ holds for all $i$. This however, comes at the cost of many iterations, since a big constant $\gamma$ causes bad condition of our problem. As a remedy we implemented a $\gamma$-homotopy like e.g. in \cite[Section 6]{CasasClasonKunisch}, where we start with $\gamma = 1$, solve the problem using the semismooth Newton method and use this solution as a starting point for an increased $\gamma$ until a solution satisfies the constraints. With a fixed $\gamma =70$ we need almost 1000 Newton steps, with the the $\gamma$-homotopy, which terminates at $\gamma =64$ in this setting, it takes 183 Newton steps.

As a comparison to the problem with only positive sources, we also solve the problem with the same reachable desired state as in Figure 4, i.e. the projection of the original desired state onto the coarse grid. Here, we also observe $\bar y (T) \approx y_d$ (with an error of size $10^{-12}$) and $\bar \varphi \approx 0 \in Q$. Furthermore the optimal control is sparse with $\supp(\bar u ^+) = \left\{ 0.5\right\}$, only consists of a positive part and its total variation is $\bar{u}^+(\bar{\Omega}) =1 < \alpha$. We fix $\gamma = 70$ and need 56 Newton steps in this case.

\begin{figure}[ht]
	\begin{center}
		\setlength{\tabcolsep}{1pt}
		\begin{tabular}{|c|c|c|c|}
			\hline	
			\raisebox{-1\height}{\includegraphics[width=\imgwidth]{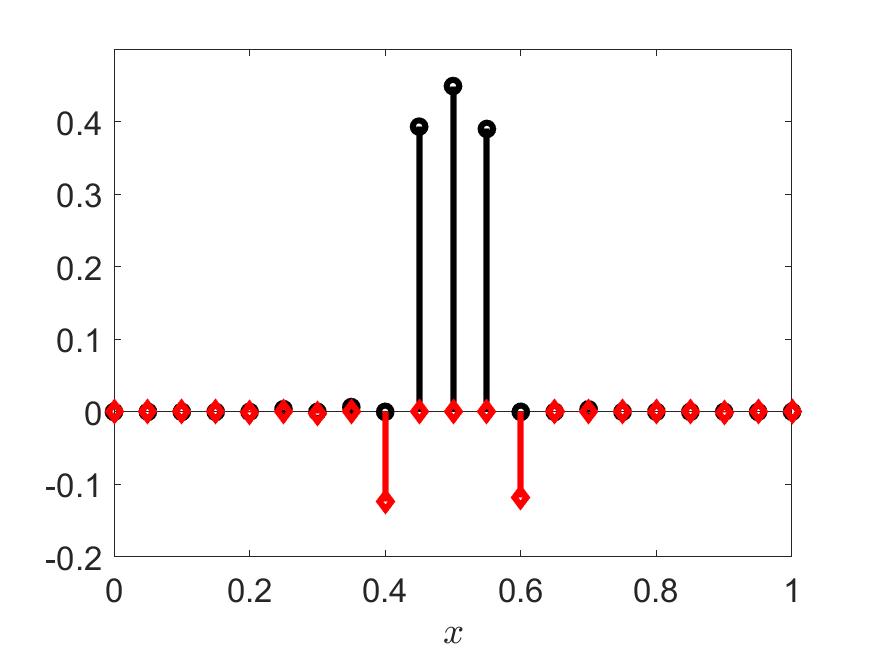}}
			&\raisebox{-1\height}{\includegraphics[width=\imgwidth]{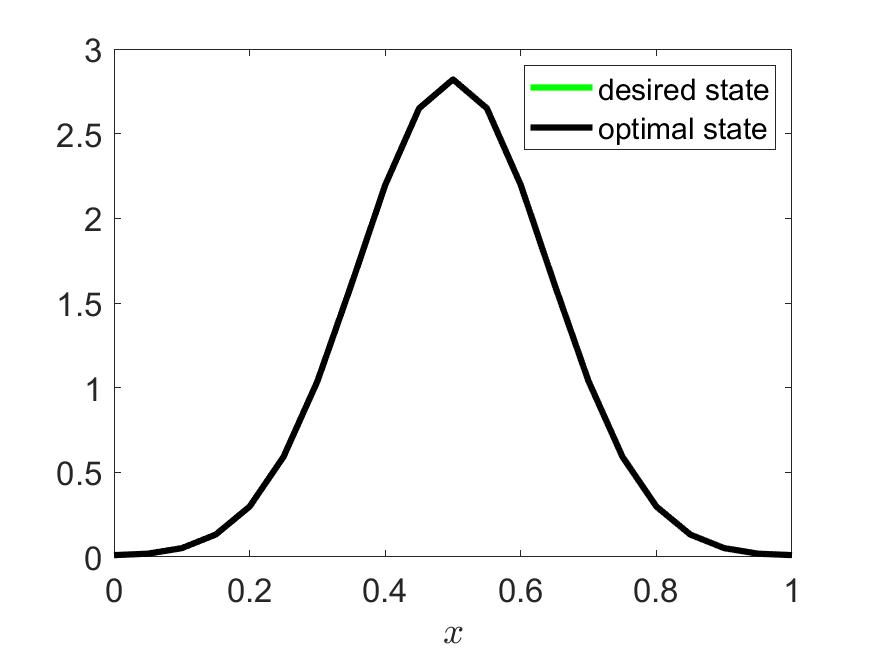}}
			&\raisebox{-1\height}{\includegraphics[width=\imgwidth]{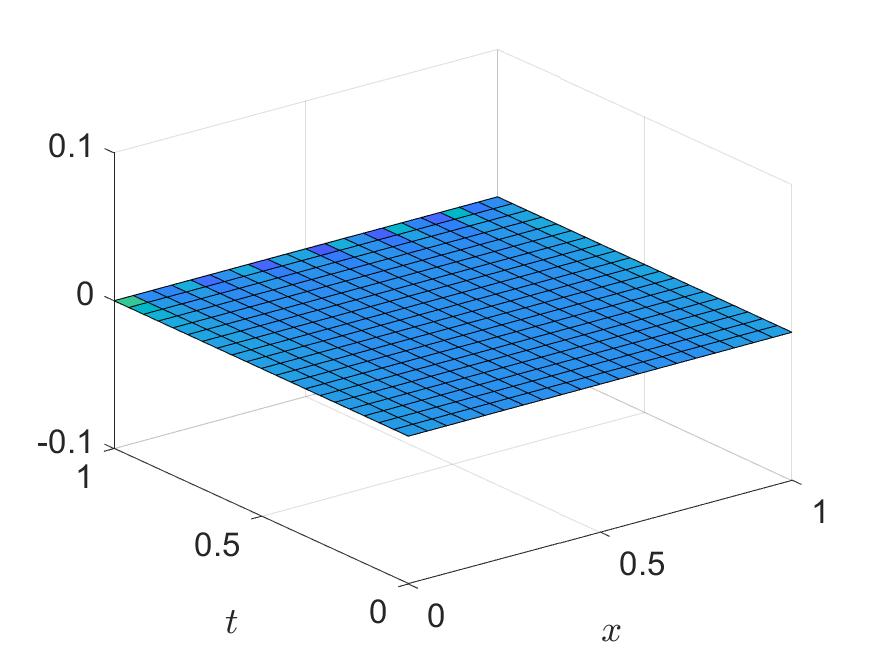}}
			&\raisebox{-1\height}{\includegraphics[width=\imgwidth]{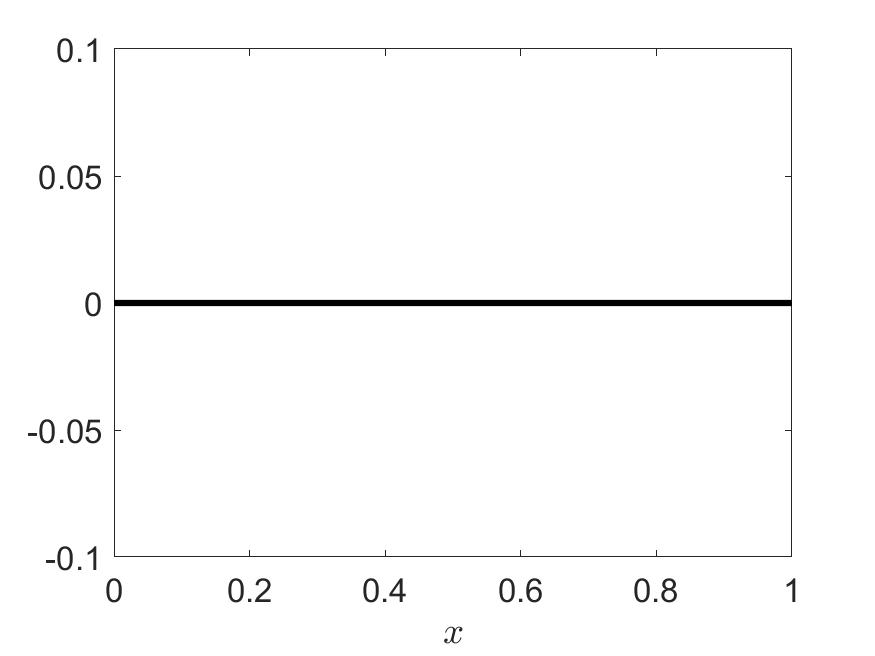}}
			\\
			\hline
			\raisebox{-1\height}{\includegraphics[width=\imgwidth]{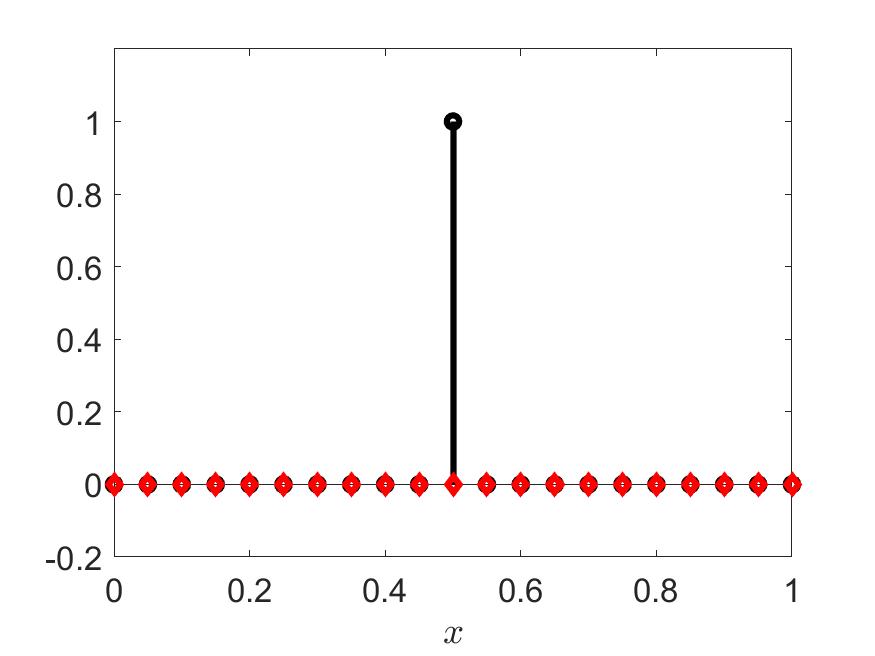}}
			&\raisebox{-1\height}{\includegraphics[width=\imgwidth]{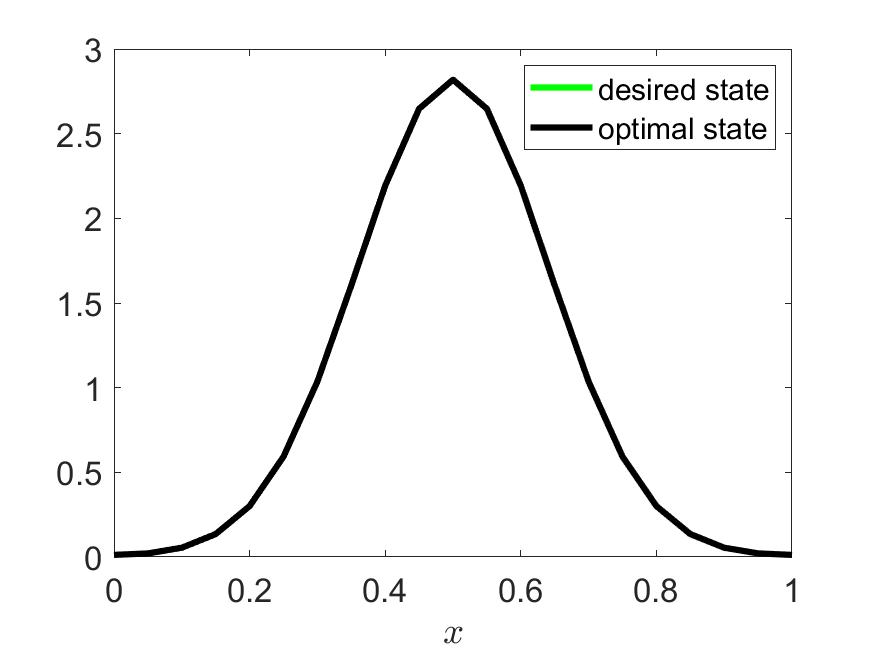}}
			&\raisebox{-1\height}{\includegraphics[width=\imgwidth]{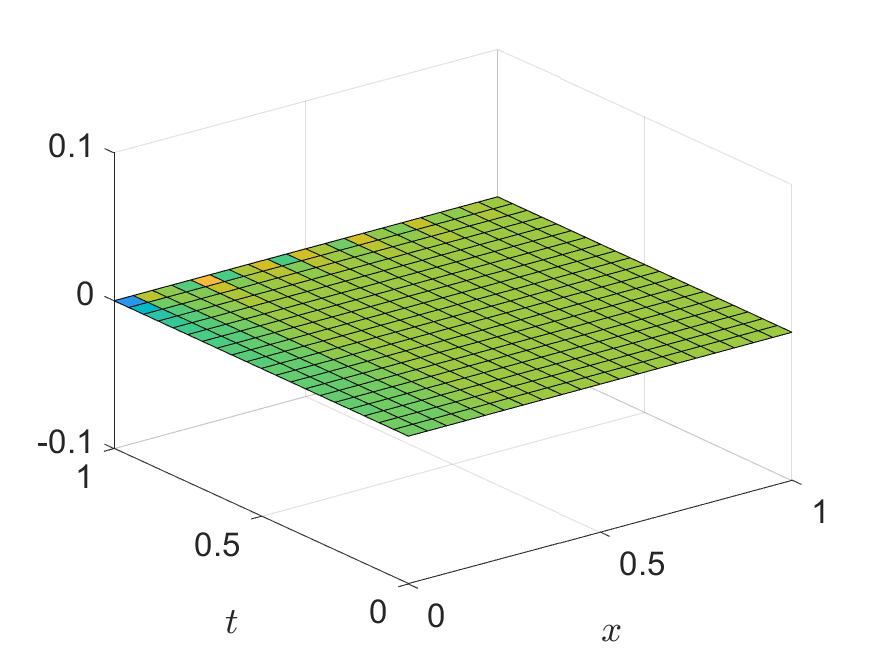}}
			&\raisebox{-1\height}{\includegraphics[width=\imgwidth]{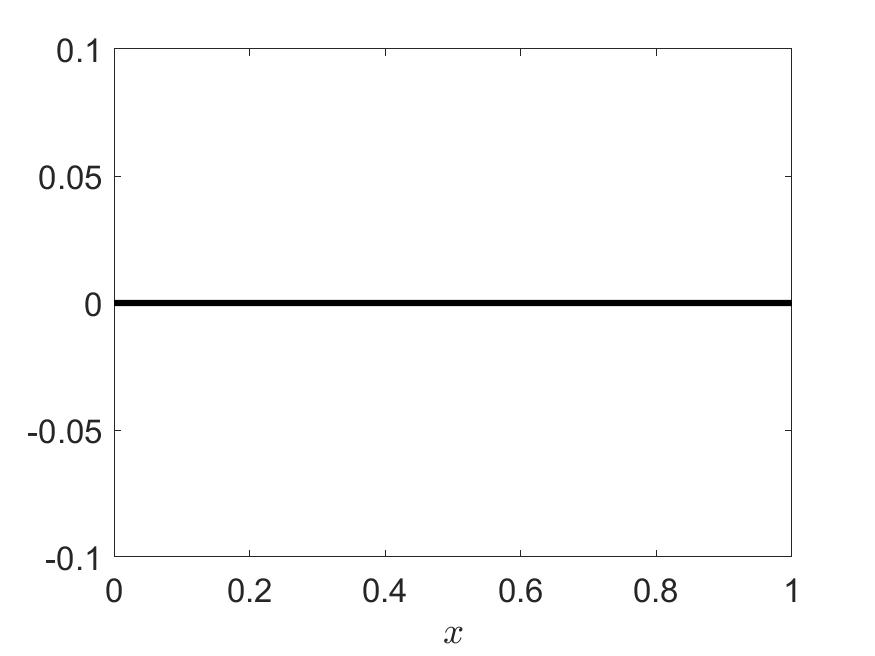}}
			\\
			\hline
		\end{tabular}
	\end{center}
	\caption{Solutions for $\alpha = 2$ with original desired state (top) and reachable desired state (bottom): from left to right: optimal control $\bar{u} = \bar{u}^+ - \bar{u}^-$ (solved with the semismooth Newton method), associated optimal state $\bar{y}$, associated adjoint $\bar{\varphi}$ on the whole space-time domain $Q$, associated adjoint $\bar{\varphi}$ at $t=0$. Terminated after 183 and 56 Newton steps, respectively.}
\end{figure}

Furthermore, we solve this case on a finer mesh (40 $\times$ 40) to compare the behavior of solutions (see Figure 7). We observe a higher iteration count: 255 Newton steps when employing a $\gamma-$homotopy, which terminates at $\gamma=64$. In fact for any example, which we solved on two different meshes the solver needed more iterations on the finer grid. This is caused by the growing condition number of the PDE solver, since it is a mapping from an initial measure control to the state at final time. We can also see a difference in the optimal controls in Figure 6 top and Figure 7, although comparable associated optimal state and adjoint are achieved. 

\begin{figure}[ht]
	\begin{center}
		\setlength{\tabcolsep}{1pt}
		\begin{tabular}{|c|c|c|c|}
			\hline
			\raisebox{-1\height}{\includegraphics[width=\imgwidth]{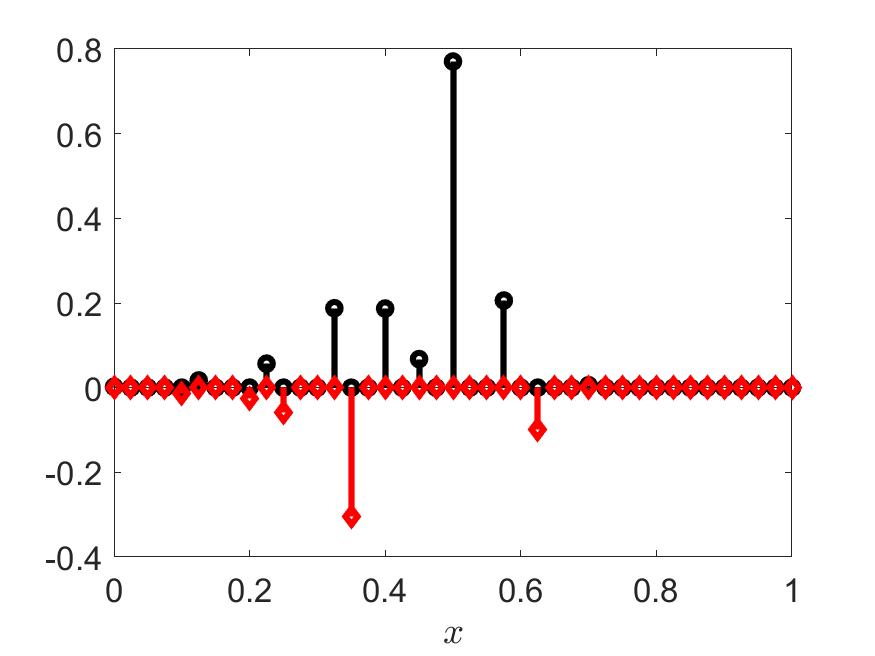}}
			&\raisebox{-1\height}{\includegraphics[width=\imgwidth]{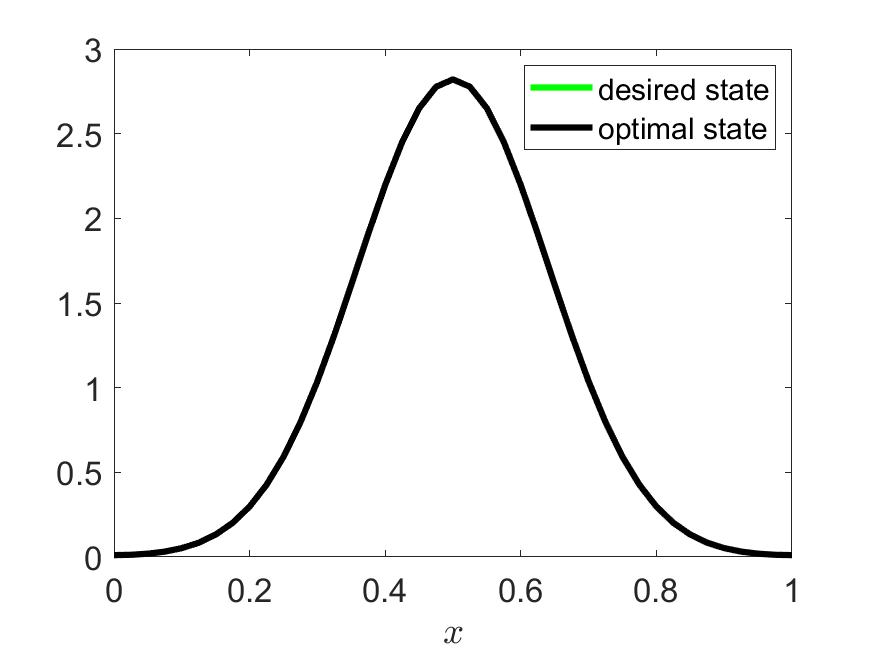}}
			&\raisebox{-1\height}{\includegraphics[width=\imgwidth]{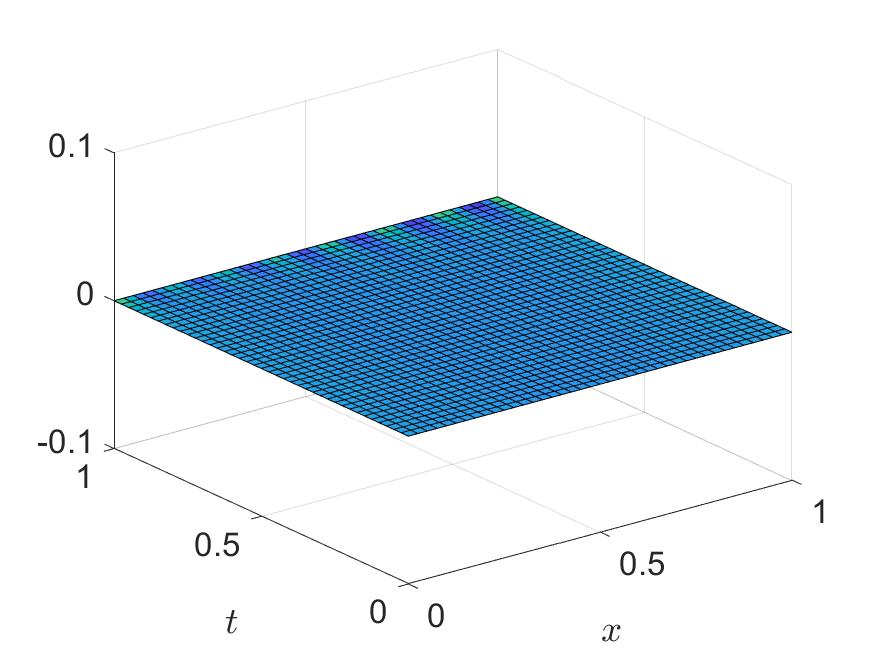}}
			&\raisebox{-1\height}{\includegraphics[width=\imgwidth]{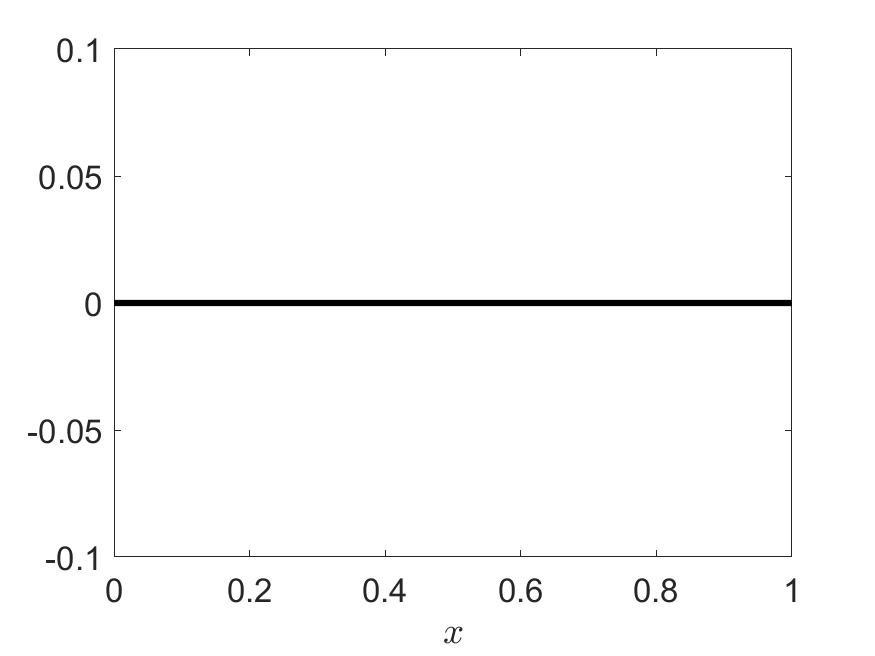}}
			\\
			\hline
		\end{tabular}
	\end{center}
	\caption{Solution for $\alpha = 2$ with original desired state on a 40 $\times$ 40 grid: from left to right: optimal control $\bar{u} = \bar{u}^+ - \bar{u}^-$ (solved with the semismooth Newton method), associated optimal state $\bar{y}$, associated adjoint $\bar{\varphi}$ on the whole space-time domain $Q$, associated adjoint $\bar{\varphi}$ at $t=0$. Terminated after 255 Newton steps.}
\end{figure}

The second example we want to look at is a measure consisting of a positive and a negative part. To generate a desired state $y_d$, we choose $u_{\operatorname{true}} = \delta_{0.3} - 0.5 \cdot \delta_{0.8}$ and $f \equiv 0$ , solve the state equation on a very fine grid ($1000 \times 1000$) and take the evaluation of the result in $t = T$ on the current grid $\Omega_h$ as desired state $y_d$ (see Figure 8). 

\begin{figure}[ht]
	\begin{center}
		\setlength{\tabcolsep}{1pt}
		\begin{tabular}{|c|c|c|}
			\hline
			\raisebox{-1\height}{ \includegraphics[width=\imgwidth]{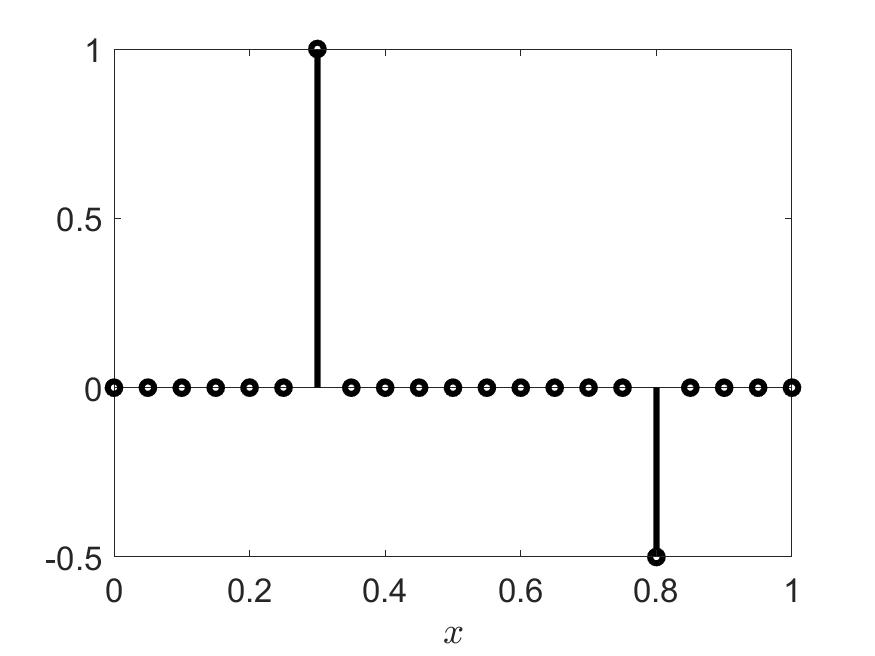}}
			&\raisebox{-1\height}{ \includegraphics[width=\imgwidth]{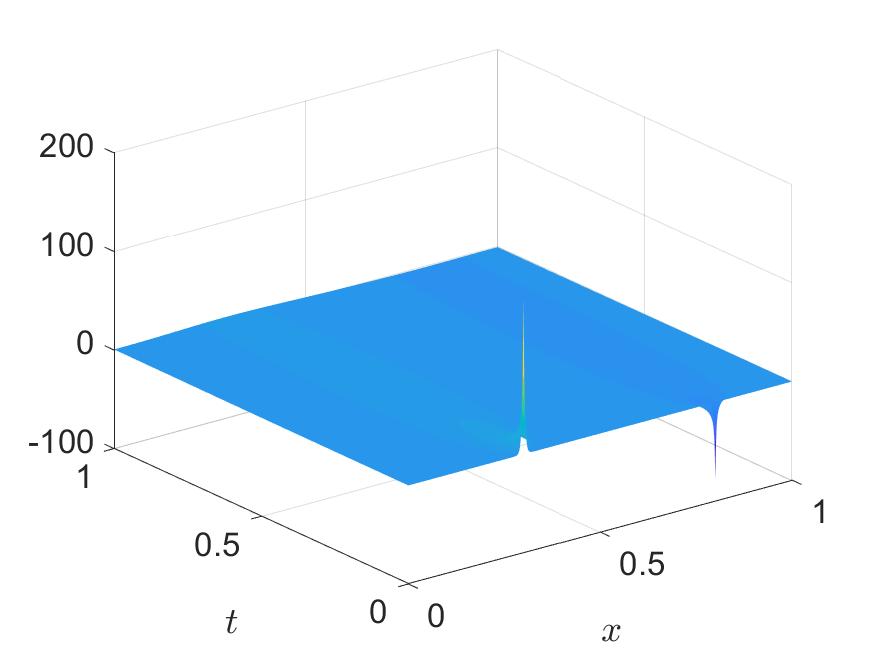}}
			&\raisebox{-1\height}{ \includegraphics[width=\imgwidth]{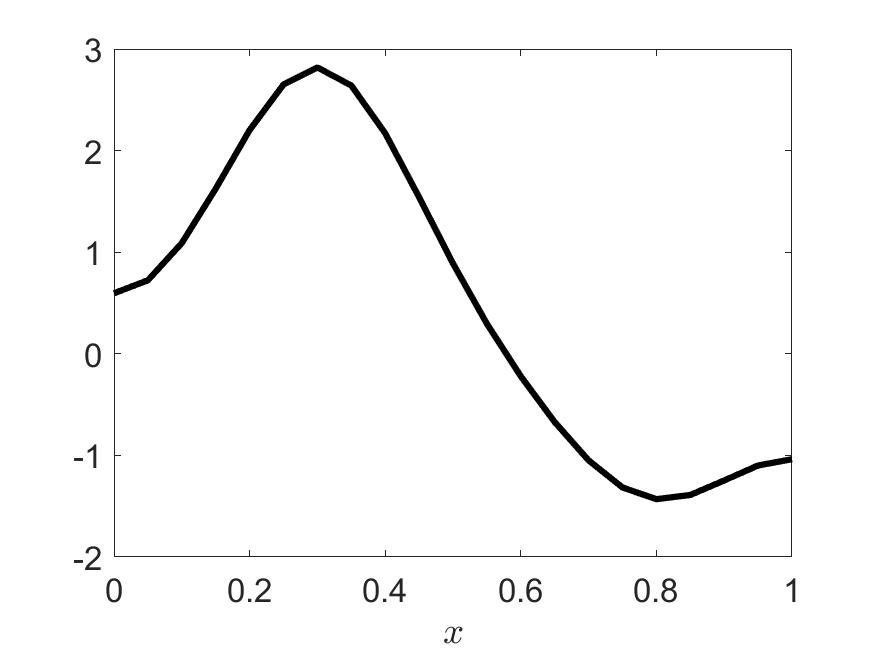}}
			\\
			\hline
		\end{tabular}
	\end{center}
	\caption{From left to right: true solution $u_{\operatorname{true}}$, associated true state $y_{\operatorname{true}}$ in $Q =[0,1] \times [0,1]$ and desired state $y_d = y_{\operatorname{true}}(T)$}	
	\label{fig:1}
\end{figure}

The first case we investigate is $\alpha = 0.15$ (see Figure 9 top). This $\alpha$ is smaller than the total variation of the true control and we observe $\bar{u}^+(\bar{\Omega}) = 0.15, \bar{u}^-(\bar{\Omega}) = 1.2929 \cdot 10^{-16}$. 
The second case we investigate is $\alpha = 1.5$ (see Figure 9 bottom). This $\alpha$ is equal to the total variation of the true control and we observe $\bar{u}^+(\bar{\Omega}) = 1.0001, \bar{u}^-(\bar{\Omega}) = 0.4999$. For both cases displayed in Figure 8 we fix $\gamma = 70$. 

\begin{figure}[ht]
	\begin{center}
		\setlength{\tabcolsep}{1pt}
		\begin{tabular}{|c|c|c|c|}
			\hline	
			\raisebox{-1\height}{\includegraphics[width=\imgwidth]{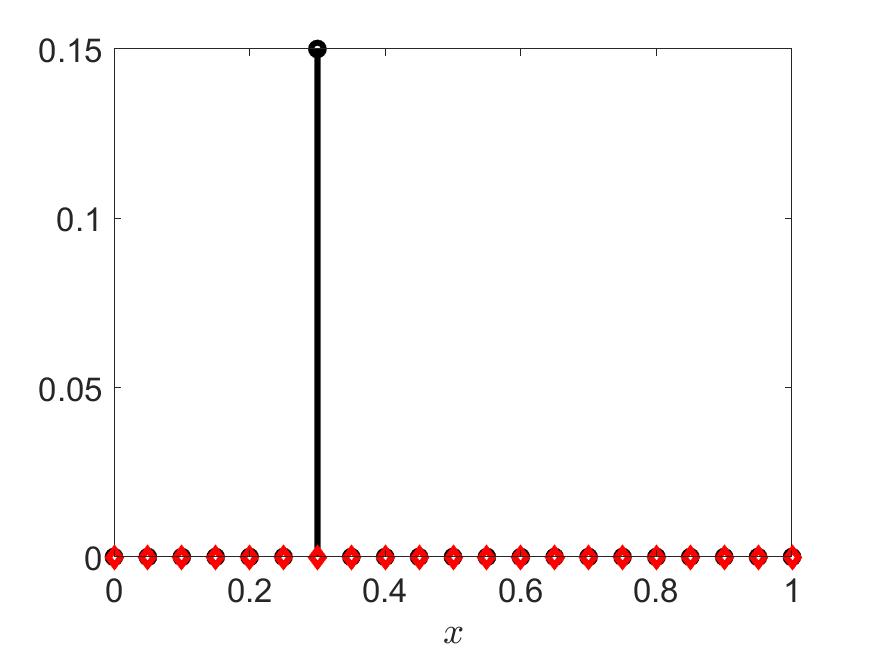}}
			&\raisebox{-1\height}{\includegraphics[width=\imgwidth]{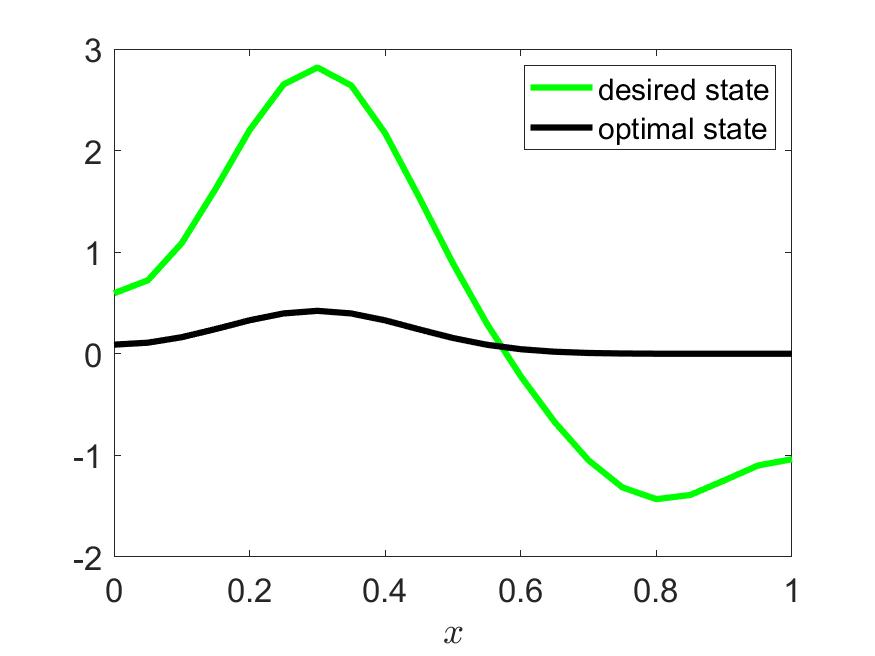}}
			&\raisebox{-1\height}{\includegraphics[width=\imgwidth]{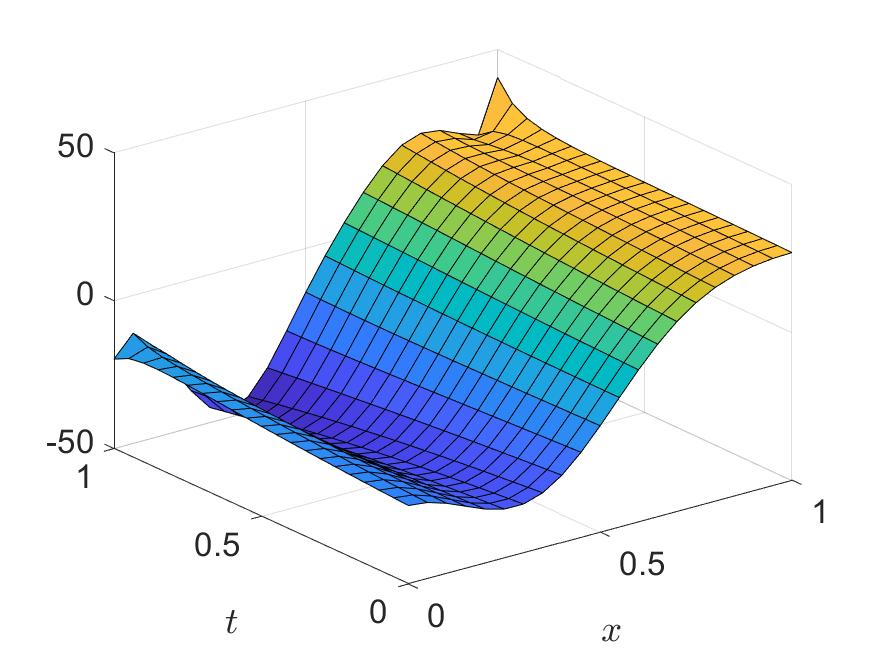}}
			&\raisebox{-1\height}{\includegraphics[width=\imgwidth]{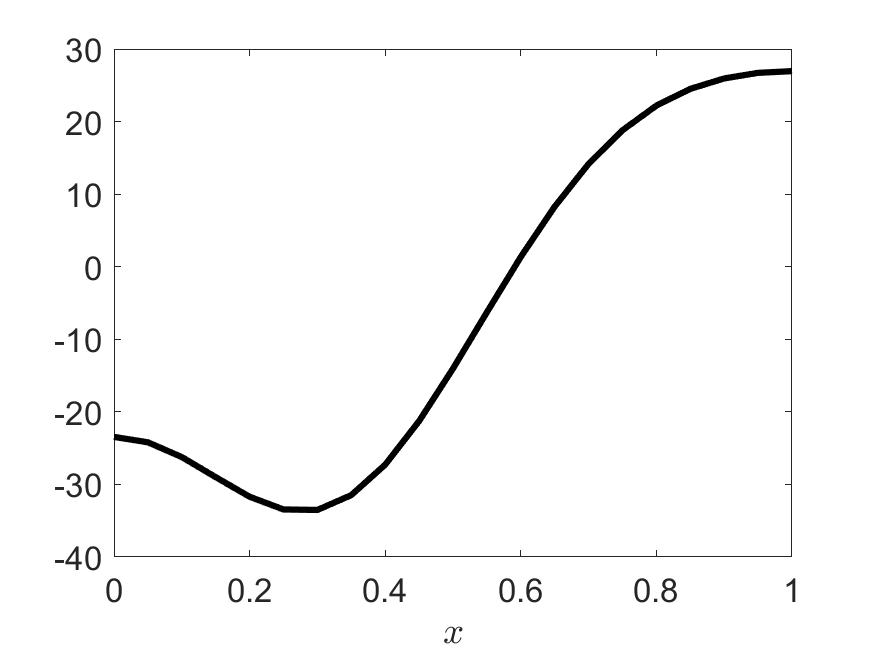}}
			\\
			\hline
			\raisebox{-1\height}{\includegraphics[width=\imgwidth]{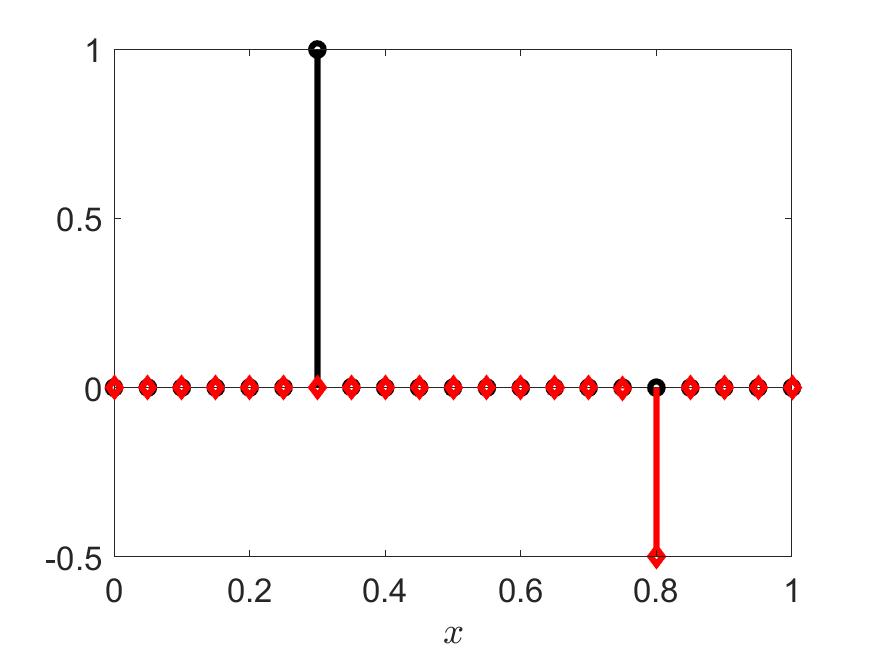}}
			&\raisebox{-1\height}{\includegraphics[width=\imgwidth]{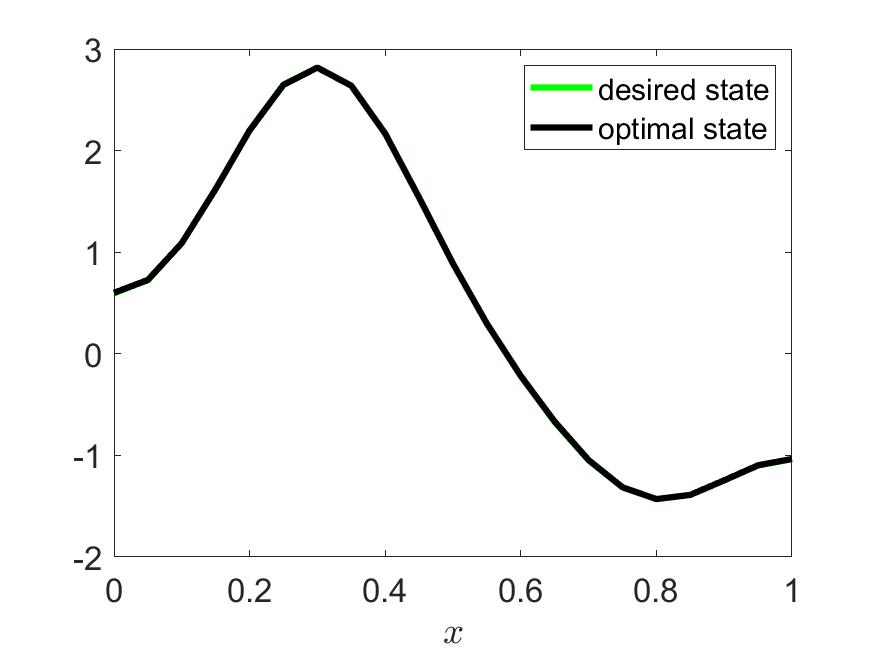}}
			&\raisebox{-1\height}{\includegraphics[width=\imgwidth]{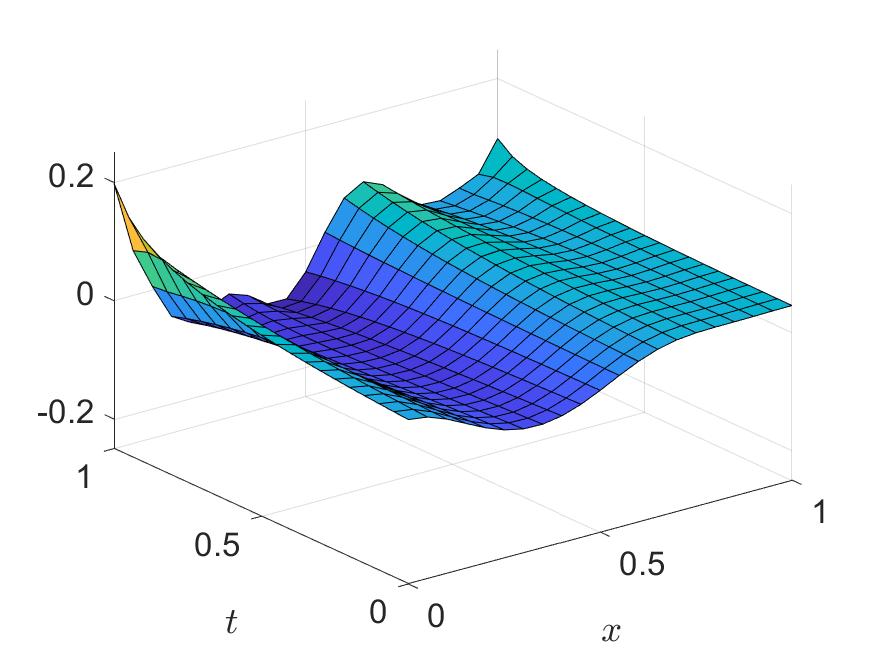}}
			&\raisebox{-1\height}{\includegraphics[width=\imgwidth]{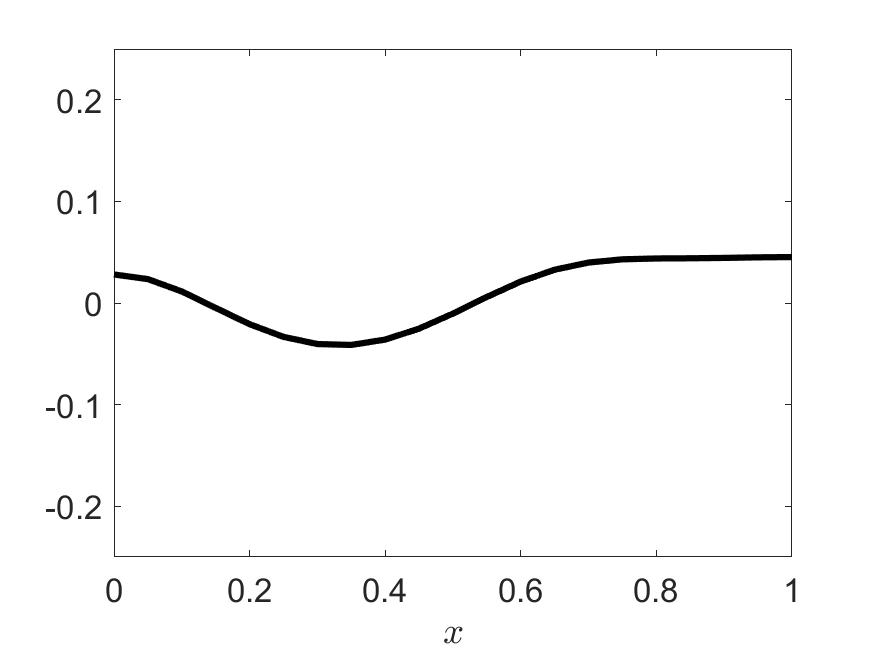}}
			\\
			\hline
		\end{tabular}
	\end{center}
	\caption{Solutions for $\alpha = 0.15$ (top) and $\alpha = 1.5$ (bottom): from left to right: optimal control $\bar{u} = \bar{u}^+ - \bar{u}^-$ (solved with the semismooth Newton method), associated optimal state $\bar{y}$, associated adjoint $\bar{\varphi}$ on the whole space-time domain $Q$, associated adjoint $\bar{\varphi}$ at $t=0$. Terminated after 29 and 44 Newton steps, respectively.}
\end{figure}

Again, we investigate as third case a setting, where $\alpha = 3 > 1.5 = \|u_{\operatorname{true}}\|_{\mathcal{M}(\bar \Omega)}$ (see Figure 10). We observe $\bar{u}^+(\bar{\Omega}) = 1.75, \bar{u}^-(\bar{\Omega}) = 1.25$. Here, $\bar y (T) \approx y_d$ (with an error of size $10^{-7}$) and $\bar \varphi \approx 0 \in Q$ hold. The optimal control fulfills the complementarity condition, but we can not observe the same sparsity that was inherited by $u_{\operatorname{true}}$. For this case we have to raise the fix $\gamma$ to 100 and the computation took over 1700 Newton steps. Hence we employ a $\gamma$-homotopy again, which terminates at $\gamma =64$ in this setting, and only need 137 Newton steps.

For comparison we project the desired state onto the coarse grid, such that it becomes reachable and then solve the problem again. Now we observe $\bar{u}^+(\bar{\Omega}) = 1, \bar{u}^-(\bar{\Omega}) = 0.5, \supp(\bar u^+) = \left\{0.3\right\}, \supp(\bar u ^-) = \left\{0.8\right\}$, which are exactly the properties of $u_{\operatorname{true}}$. Furthermore we see $\bar y (T) \approx y_d$ (with an error of size $10^{-14}$) and $\bar \varphi \approx 0 \in Q$. We observe a reduction of Newton steps needed - the computation took 20 Newton steps with fixed $\gamma = 100$.

\begin{figure}[ht]
	\begin{center}
		\setlength{\tabcolsep}{1pt}
		\begin{tabular}{|c|c|c|c|}
			\hline	
			\raisebox{-1\height}{\includegraphics[width=\imgwidth]{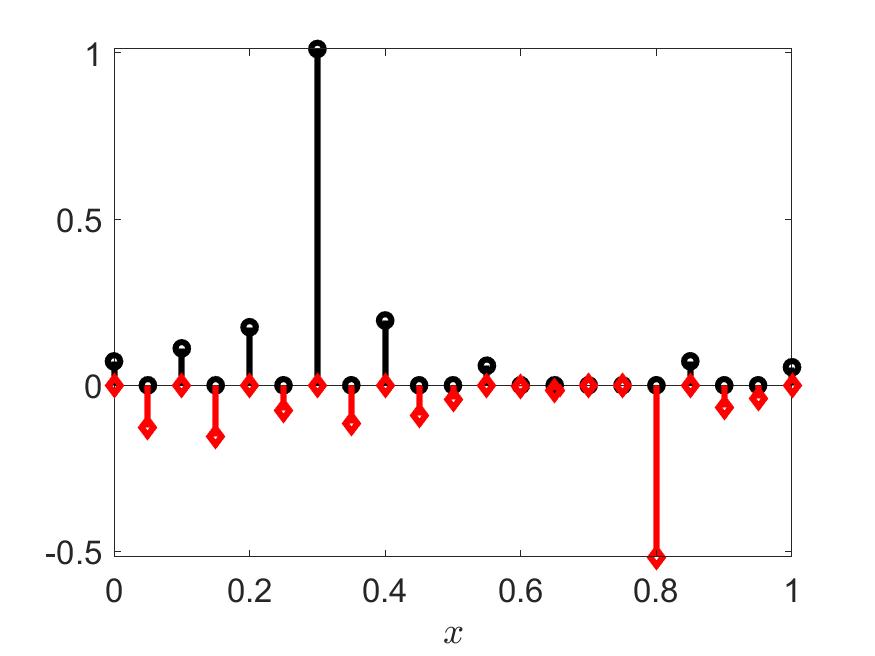}}
			&\raisebox{-1\height}{\includegraphics[width=\imgwidth]{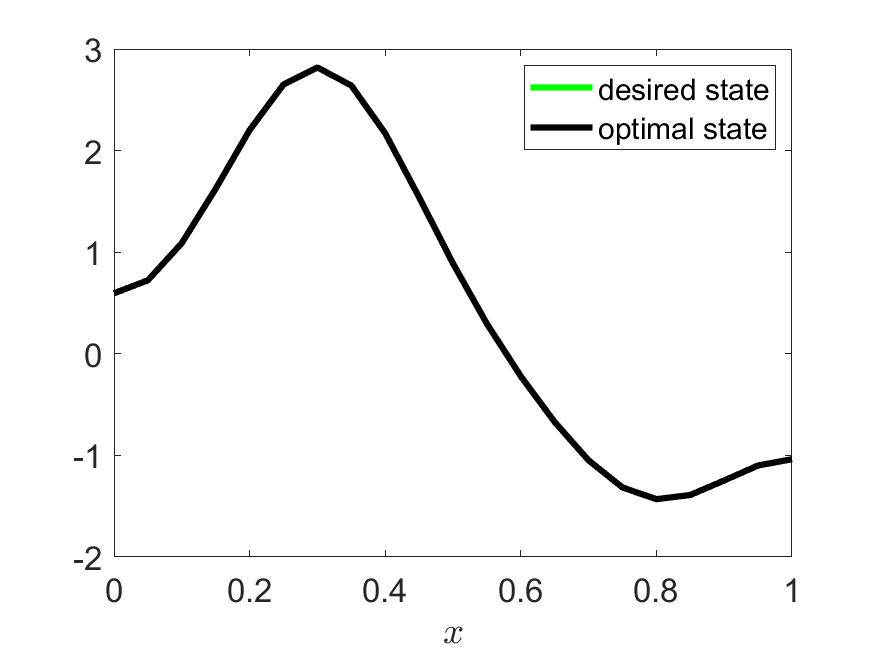}}
			&\raisebox{-1\height}{\includegraphics[width=\imgwidth]{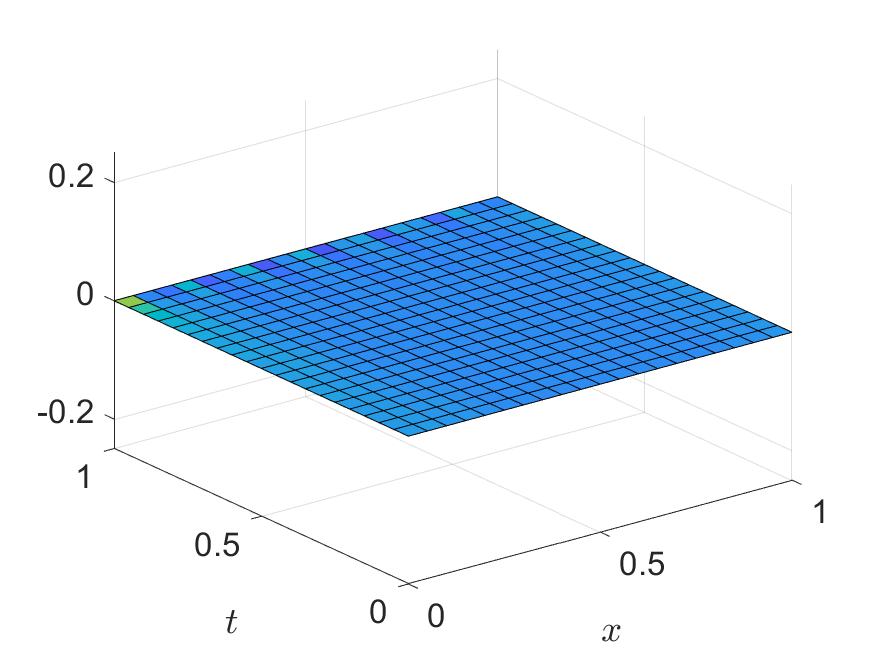}}
			&\raisebox{-1\height}{\includegraphics[width=\imgwidth]{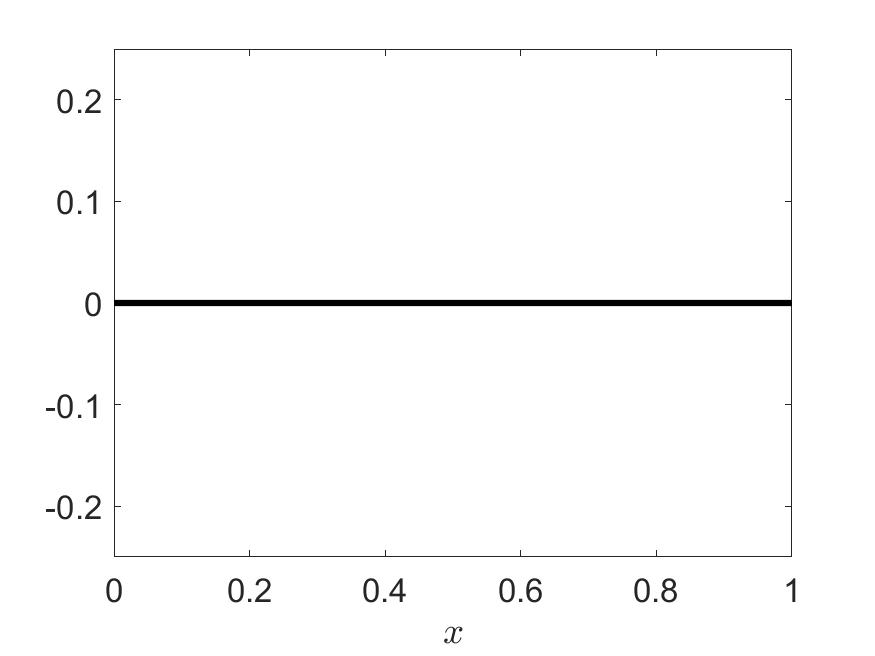}}
			\\
			\hline
			\raisebox{-1\height}{\includegraphics[width=\imgwidth]{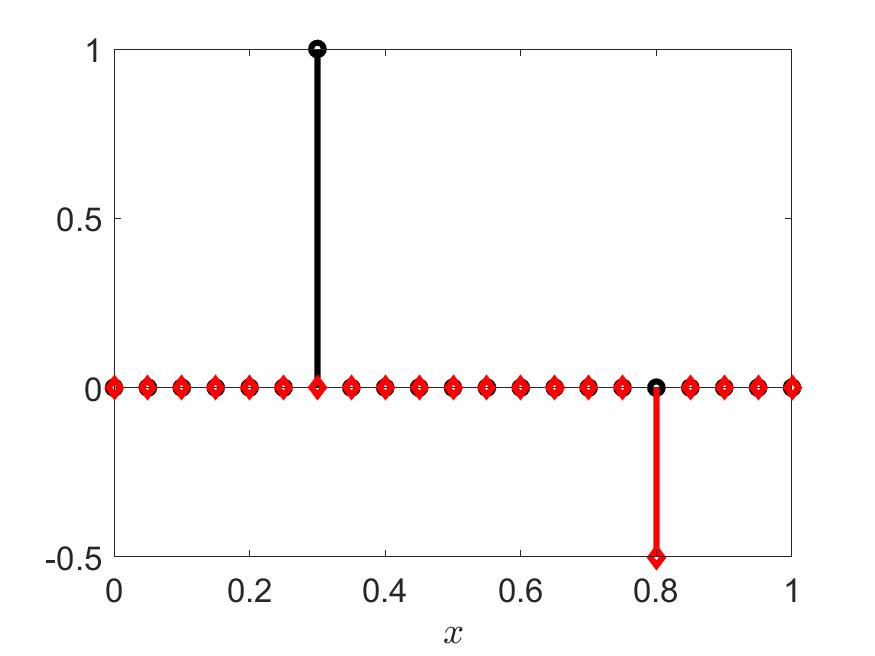}}
			&\raisebox{-1\height}{\includegraphics[width=\imgwidth]{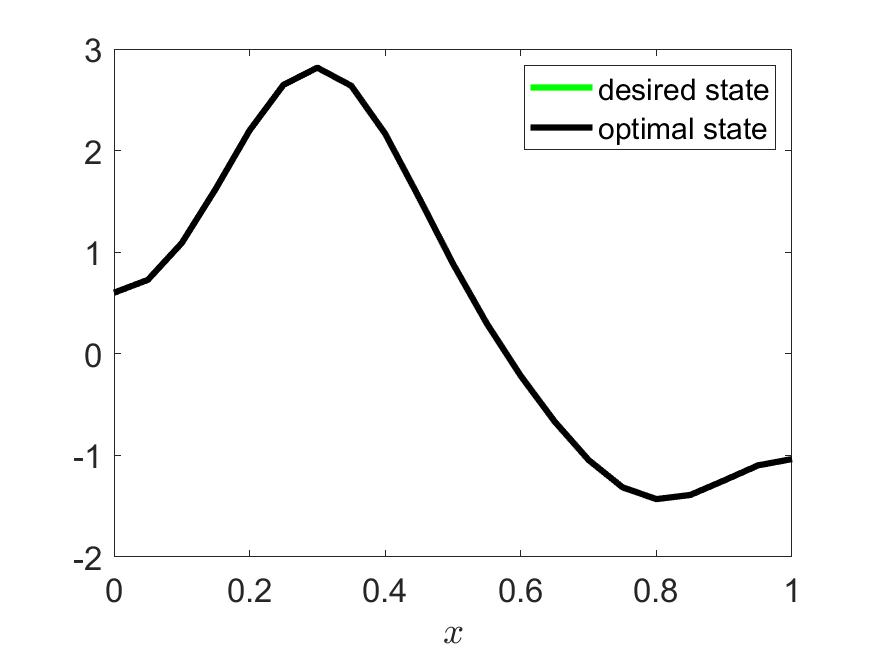}}
			&\raisebox{-1\height}{\includegraphics[width=\imgwidth]{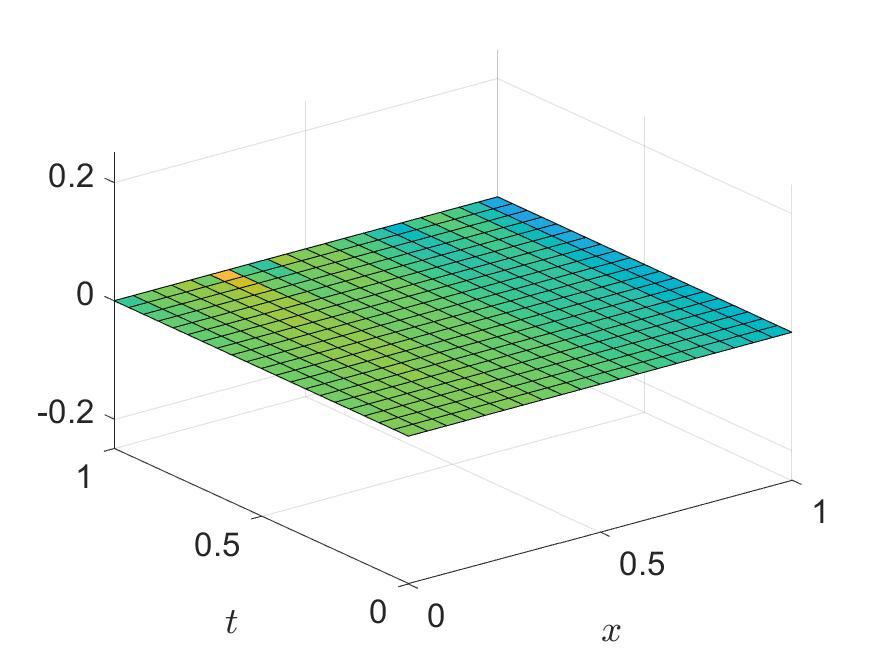}}
			&\raisebox{-1\height}{\includegraphics[width=\imgwidth]{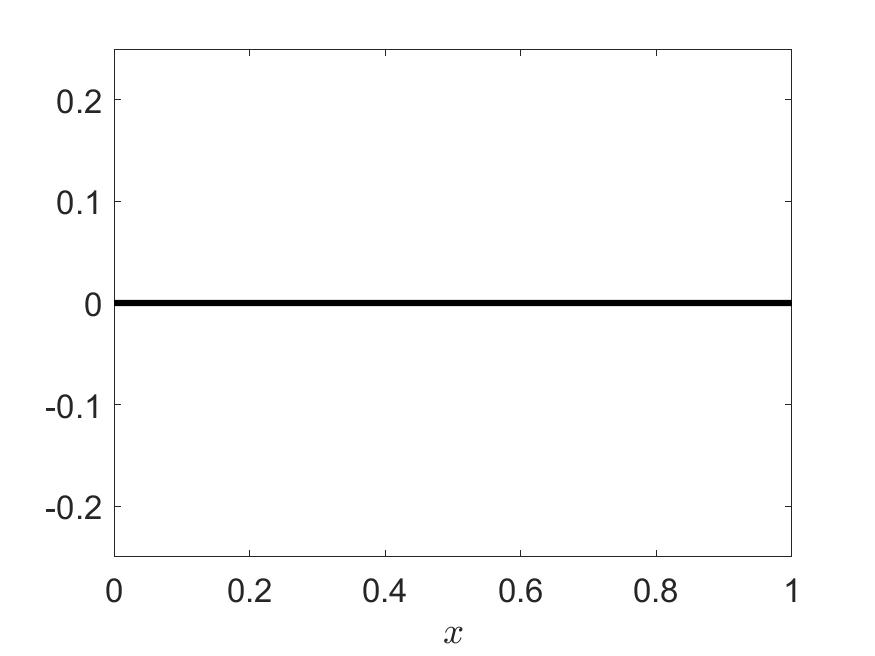}}
			\\
			\hline
		\end{tabular}
	\end{center}
	\caption{Solutions for $\alpha = 3$ with original desired state (top) and reachable desired state (bottom): from left to right: optimal control $\bar{u} = \bar{u}^+ - \bar{u}^-$ (solved with the semismooth Newton method), associated optimal state $\bar{y}$, associated adjoint $\bar{\varphi}$ on the whole space-time domain $Q$, associated adjoint $\bar{\varphi}$ at $t=0$. Terminated after 137 and 20 Newton steps.}
\end{figure}

	\textbf{Acknowledgment:} We acknowledge the fruitful discussions with both Eduardo Casas and Karl Kunisch, which inspired this work.


\end{document}